%% file: LSrec.tex
\documentclass[a4paper, 12pt]{amsart}
\usepackage[utf8]{inputenc}
\usepackage[colorlinks]{hyperref}
\usepackage{graphicx}
\usepackage{latexsym}
\usepackage{amsmath}
\usepackage{amssymb}
\usepackage{accents}
\usepackage{epsf}
\usepackage{epic}
\usepackage{supertabular}
\usepackage{rotating}
\usepackage{multicol}
\usepackage{multirow}
\usepackage{amsthm}
\usepackage{array}
\usepackage{textcomp}
\usepackage[shortlabels]{enumitem}
\usepackage[autostyle]{csquotes}

\usepackage{epsfig}
\usepackage{tikz}

\newtheorem{thm}{Theorem}[section]
\newtheorem{lem}[thm]{Lemma}
\newtheorem{cor}[thm]{Corollary}

\theoremstyle{definition}
\newtheorem{defn}[thm]{Definition}
\newtheorem{ex}[thm]{Example}

\newtheorem{rem}[thm]{Remark}

\theoremstyle{remark}

\newcommand{\gauss}[3]{\genfrac{[}{]}{0pt}{}{#1}{#2}_{#3}}
\newcommand{\gaussm}[2]{\genfrac{[}{]}{0pt}{}{#1}{#2}}
\newcommand{\N}{\mathbb{N}}

\newcommand{\Z}{\mathbb{Z}}

\DeclareMathOperator{\GF}{GF}
\DeclareMathOperator{\LS}{LS}
\DeclareMathOperator{\cm}{cm}
\DeclareMathOperator{\PGammaL}{P\Gamma L}
\DeclareMathOperator{\PGL}{PGL}
\DeclareMathOperator{\GL}{GL}

\makeatletter
\let\@@mod\mod
\DeclareRobustCommand{\mod}{\@ifstar\@mods\@@mod}
\def\@mods#1{\mkern4mu({\operator@font mod}\mkern 6mu#1)}
\makeatother

\title{Large sets of subspace designs}
\author[M. Braun]{Michael Braun}
\address{Michael Braun\\Hochschule Darmstadt\\Schöfferstr. 8b\\D-64295 Darmstadt\\Germany}
\email{michael.braun@h-da.de}

\author[M. Kiermaier]{Michael Kiermaier}
\address{Michael Kiermaier\\Mathematisches Institut\\Universität Bayreuth\\D-95440 Bayreuth\\Germany} 
\email{michael.kiermaier@uni-bayreuth.de} 
\urladdr{http://www.mathe2.uni-bayreuth.de/michaelk/}

\author[A. Kohnert]{Axel Kohnert}

\author[R. Laue]{Reinhard Laue}
\address{Reinhard Laue\\Institut für Informatik\\Universität Bayreuth\\D-95440 Bayreuth\\Germany} 
\email{laue@uni-bayreuth.de} 

\date{\today}
\dedicatory{To the memory of Axel Kohnert 1962--2013}

\begin{document}
\begin{abstract}
\input{abstract.tex}
\end{abstract}

\maketitle

\input{introduction.tex}
\input{preliminaries.tex}
\input{decompositions.tex}
\input{n-t-partitionable.tex}
\input{examples_halvings.tex}
\input{infinite_series.tex}
\input{conclusion.tex}

\section*{Acknowledgement}
The authors would like to acknowledge the financial support provided by COST -- \emph{European Cooperation in Science and Technology}.
The authors are members of the Action IC1104 \emph{Random Network Coding and Designs over GF(q)}.

\end{document}

%% file: abstract.tex
In this article, three types of joins are introduced for subspaces of a vector space.
Decompositions of the Graßmannian into joins are discussed.
This framework admits a generalization of large set recursion methods for block designs to subspace designs.

We construct a $2$-$(6,3,78)_5$ design by computer, which corresponds to a halving $\LS_5[2](2,3,6)$.
The application of the new recursion method to this halving and an already known $\LS_3[2](2,3,6)$ yields two infinite two-parameter series of halvings $\LS_3[2](2,k,v)$ and $\LS_5[2](2,k,v)$ with integers $v\geq 6$, $v\equiv 2\mod* 4$ and $3\leq k\leq v-3$, $k\equiv 3\mod* 4$.

Thus in particular, two new infinite series of nontrivial subspace designs with $t = 2$ are constructed.
Furthermore as a corollary, we get the existence of infinitely many nontrivial large sets of subspace designs with $t = 2$.

%% file: introduction.tex
\section{Introduction}
\subsection{History}
Due to the connection to network codes, there has been a growing interest in $q$-analogs of block designs (subspace designs) lately.
The earliest reference is \cite{Cameron74}.
However, the idea is older, since it is stated that \enquote{Several people have observed that the concept of a $t$-design can be generalised [...]}.
They have also been mentioned in a more general context in \cite{Delsarte-1976}.
An introduction can be found in \cite[Day~4]{suzuki-fivedays}.

The first nontrivial subspace design with $t=2$ has been constructed in \cite{Thomas} and the first one with $t=3$ in \cite{BKL05}.
More constructions based on the method of \cite{Miyakawa-Munemasa-Yoshiara-1995,BKL05} have been presented in \cite{MBraun05, SBraunDiplomarbeit}.
In \cite{BEOVW13}, the first $q$-analog of a Steiner system with $t\geq 2$ has been constructed by applying the Kramer-Mesner method described in \cite{Miyakawa-Munemasa-Yoshiara-1995,BKL05}.
Furthermore, in \cite{FLV} it was shown that simple $t$-designs exist for every value of $t$.
This is a $q$-analog of Teirlinck's theorem \cite{T}, however with the difference that the proof in \cite{FLV} is not constructive.

The first large set of subspace designs was constructed in \cite{MBraun05} and a further one in \cite{BKOW}.
In \cite{trivial}, derived and residual subspace designs, and in \cite{Kiermaier-Pavcevic-2014}, intersection numbers for subspace designs have been studied.

To our knowledge, besides \cite{FLV} the only known infinite series of nontrivial subspace designs with $t\geq 2$ so far are the following:
In \cite{Thomas} a series of $2$-designs was constructed for $q = 2$ and generalized to all prime powers $q$ in \cite{Suzuki-1} and \cite{Suzuki-2}.
Based on these designs, the recursive construction in \cite{Itoh} provides further $2$-designs.

\subsection{Overview}
In this article, we will construct two new infinite two-parameter series of subspace designs with $t = 2$.
More precisely, an infinite series of halvings $\LS_3[2](2,k,v)$ and $\LS_5[2](2,k,v)$ with integers $v\geq 6$, $v\equiv 2\mod* 4$ and $3\leq k\leq v-3$, $k\equiv 3\mod* 4$ will be given.%
\footnote{For an explanation of the symbol $\LS_q[N](t,k,v)$, see Definition~\ref{def:ls}.}

The first step is the construction of the smallest members of both series ($\LS_3[2](2,3,6)$ and $\LS_5[2](2,3,6)$).
In the first case, this large set is already known~\cite{MBraun05}.
In the second case, it is constructed by computer using the method of Kramer and Mesner \cite{Kramer-Mesner}, prescribing some subgroup of the normalizer of a Singer cycle as a group of automorphisms.

To extend both halvings to an infinite series, recursion methods for large sets of subspace designs will be developed.
For ordinary block designs, this idea goes back to Teirlinck \cite{T89}.
Our approach is based on decompositions of the Graßmannian into \emph{joins} and can be seen as a $q$-analog of the strategy of Ajoodani-Namini and Khosrovshahi \cite{Ajoodani91, Ajoodani94, Ajoodani96}.
A survey can be found in \cite{Khosrovshahi-TayfehRezaie-2009}.

\subsection{Outline}
Section~\ref{sect:prelim} provides the required fundamentals about the subspace lattice, canonical matrices of subspaces, subspace designs and their large sets.
In Section~\ref{sect:decompositions}, the ordinary join, the covering join and the avoiding join of subspaces are introduced.
The theory is developed in a basis-free manner.
Whenever possible, moreover a representation based on canonical matrices is given, which leads to a connection to paths in $q$-grid graphs.
As an important component of the later constructions, decompositions of the Graßmannian into the three types of joins are studied.

The next Section~\ref{sect:n-t-partitionable} introduces $(N,t)$-partitionable sets of subspace designs.
It may be understood as a weakening of the notion of a large set.
The Basic Lemma~\ref{lem:basic} states that the property of being $(N,t)$-partitionable is inherited from subspaces to joins of them.
Together with the decompositions of Section~\ref{sect:decompositions}, it provides a fairly general machinery for the recursive construction of large sets, which has proven quite powerful for ordinary block designs.

Based on computational results, in Section~\ref{sect:halving_computer_constructions} two theorems about halvings with the parameters $\LS_3[2](2,3,6)$ and $\LS_5[2](2,3,6)$ are proven.
In Section~\ref{sect:series}, the recursive construction method of Section~\ref{sect:n-t-partitionable} is applied to these two halvings.
The construction is carried out in two steps, the first one based on a decomposition into avoiding joins, the second one based on a decomposition into covering joins.
The result consists in two new infinite two-parameter series of halvings in Corollary~\ref{cor:halving_series}, which are also infinite series of nontrivial subspace designs with $t = 2$.
The article is concluded in Section~\ref{sect:conclusion} with a few open questions arising from the present work.

\subsection{Dedication}
This paper is dedicated to the memory of our estimated friend and colleague Axel Kohnert.
Axel has passed away on 11 Dec.\ 2013 in the aftermath of a tragic accident in Oct.\ 2013 at the age of 51.
He was one of the initiators of the research of large sets of subspace designs.
Investigating the recursion method by Ajoodani-Namini and Khosrovshahi for applicability in the $q$-analog situation, he developed the decomposition technique based on paths in the $q$-grid graph found in Section~\ref{sect:decompositions}.
Sadly, it was not granted to him to witness the full consequences of his idea.

%% file: preliminaries.tex
\section{Preliminaries}
\label{sect:prelim}
If not specified otherwise, $q\neq 1$ will always be a prime power, $v$ a nonnegative integer and $V$ a vector space over $\GF(q)$ of dimension $v$.

\subsection{The subspace lattice}
For an integer $k$, the set of all subspaces of $V$ of dimension $k$ is known as the \emph{Graßmannian} and will be denoted by $\gauss{V}{k}{q}$.%
\footnote{For $k < 0$ and $k > \dim(V)$, this implies $\gauss{V}{k}{q} = \emptyset$.}
For simplicity, its elements will be called \emph{$k$-subspaces}.
The size of $\gauss{V}{k}{q}$ is given by the \emph{Gaussian binomial coefficient}
\[
\gauss{v}{k}{q}
= \prod_{i = 0}^{k-1} \frac{q^{v - i} - 1}{q^{i+1} - 1}
= \begin{cases}
	\frac{(q^v - 1)(q^{v-1} - 1)\cdot\ldots\cdot(q^{v - k + 1} - 1)}{(q-1)(q^2 - 1)\cdot\ldots\cdot (q^k - 1)} & \text{if } k\in\{0,\ldots,v\}\text{;} \\
     0 & \text{otherwise.}
\end{cases}
\]
The complete subspace lattice of $V$ will be denoted by $\mathcal{L}(V)$.
There are good reasons to interpret the subspace lattice $\mathcal{L}(V)$ as the $q$-analog of the subset lattice $\mathcal{L}(X)$ where $X$ is a set of size $\#X = \dim V$ \cite{Tits,Goldman-Rota-1970-StudApplM49[3]:239-258,Cohn-2004}.

The $v\times v$ unit matrix will be denoted by $E_v$.
Furthermore, we will denote the standard basis of $\GF(q)^v$ by $e_1,\ldots,e_v\in\GF(q)^v$ and for $i\in\{0,\ldots,v\}$ we will use the notation $I_i = \{v-i+1,\ldots,v\}$ and $V_i = \langle e_j \mid j \in I_i\rangle$.
Note that $\dim(V_i) = \#I_i = i$.

The lattice $\mathcal{L}(V)$ is modular.
We will make use of the modularity law, which states that for all $A,B,C\in\mathcal{L}(V)$ with $A \leq C$,
\[
	A + (B\cap C) = (A + B)\cap C\text{.}
\]
In contrast to the subset lattice, $\mathcal{L}(V)$ is not distributive for $v\geq 2$.

By the fundamental theorem of projective geometry, for $v\geq 3$ the automorphism group of $\mathcal{L}(V)$ is given by the natural action of $\PGammaL(V)$ on $\mathcal{L}(V)$.
Furthermore, $\mathcal{L}(V)$ is self-dual.
An antiautomorphism of $\mathcal{L}(V)$ will be denoted by $\perp$.
Since two antiautomorphisms of $\mathcal{L}(V)$ differ only by an automorphism of $\mathcal{L}(V)$, in coordinate-free settings the exact choice does not really matter.
For a concrete construction, pick any non-singular bilinear form $\beta$ on $V$, and set
\[
    U^\perp = \{x\in V \mid \beta(x,y) = 0 \text{ for all }y\in U\}\text{.}
\]

\subsection{The reduced row echelon form}
In this section, $F$ denotes a field.
\begin{defn}
A $(k\times v)$-matrix $A = (a_{ij})_{i\in\{1,\ldots,k\},j\in\{1,\ldots,v\}}$ over $F$ is said to be in \emph{reduced (left) row echelon form} if there is an increasing integer sequence $1 \leq \pi_1 < \ldots < \pi_k\leq v$ of \emph{pivot positions} such that for each $i\in\{1,\ldots,k\}$, the $\pi_i$-th column of $A$ is the $i$-th standard vector in $F^k$ and the \emph{pivot entry} $a_{i,\pi_i} = 1$ is the first non-zero entry in the $i$-th row.
In this case, the set of pivot positions will be denoted by $\pi(A)$.
\end{defn}

The pivot positions of a matrix in reduced row echelon form are uniquely determined, so $\pi(A)$ is well-defined.
Each subspace $U \leq F^v$ is the row space of a unique matrix in reduced row echelon form. (The zero space being the row space of the somewhat artificial ($0\times v$)-matrix.)
This matrix will be called the \emph{canonical matrix} of $U$ and denoted by $\cm(U)$.
Thus, the reduced row echelon form provides a convenient way for the representation of subspaces.
The mapping
\[
    \pi : (\mathcal{L}(F^v),\leq) \to (\mathcal{L}(\{1,\ldots,v\}), \subseteq)
\]
is order-preserving.
For more details, see \cite[Sect.~2.2]{HKK-2013}.

The importance of the subspaces $V_i$ comes from the following easily checked property:
\begin{lem}
	\label{lem:Vi_property}
	Let $i\in\{0,\ldots,v\}$ and $U \leq \GF(q)^v$.
	The canonical matrix of $U$ has a unique block decomposition
	\[
		\cm(U) = \begin{pmatrix}
			A & B \\
			\mathbf{0} & C
		\end{pmatrix}
	\]
	where $A$ and $C$ are in reduced row echelon form, $A$ has $v-i$ columns and $C$ has $i$ columns.  
	We have
	\[
		\cm(U \cap V_i) = \begin{pmatrix}\mathbf{0} & C\end{pmatrix}
		\qquad\text{and}\qquad
		\cm(U + V_i) = \begin{pmatrix}A & \mathbf{0}\\\mathbf{0} & E_i\end{pmatrix}\text{.}
	\]
\end{lem}

When taking duals, it is natural to switch from the reduced left row echelon to the reduced \emph{right} row echelon form:

\begin{lem}
	\label{lem:cm_dual}
	Let $U\leq V$ and $U^\perp$ the dual subspace with respect to the standard bilinear form $\left\langle(x_1,\ldots,x_v), (y_1,\ldots,y_v)\right\rangle = x_1 y_1 + \ldots + x_v y_v$ on $V$.
	Let $\cm^\perp(U^\perp)$ be the unique generator matrix of $U^\perp$ in \emph{right} row echelon form.
	The positions of the pivot columns of $\cm^\perp(U^\perp)$ are given by $\{1,\ldots,v\}\setminus\pi(U)$.
	The non-pivot columns of $\cm^\perp(U^\perp)$ are given by the columns of $-A^\top$, where $A$ denotes the matrix consisting of the non-pivot columns of $\cm(U)$.
\end{lem}

\subsection{Subspace designs}

The following definition is the $q$-analog of an ordinary set-theoretic design:
\begin{defn}
A pair $(V,\mathcal{B})$ with $\mathcal{B} \subseteq \gauss{V}{k}{q}$ is called a \emph{$t$-$(v,k,\lambda)_q$ subspace design}, if for each $T\in\gauss{V}{t}{q}$ there are exactly $\lambda$ elements of $\mathcal{B}$ containing $T$.
\end{defn}

For a $t$-$(v,k,\lambda)_q$ subspace design $(V,\mathcal{B})$, its \emph{dual design} $(V,\mathcal{B})^\perp = (V,\mathcal{B}^\perp)$ with $\mathcal{B}^\perp = \{B^\perp \mid B\in\mathcal{B}\}$ is a subspace design with the parameters $t$-$(v,v-k,\lambda\cdot\gauss{v-k}{t}{q}/\gauss{k}{t}{q})_q$ \cite[Lemma~4.2]{Suzuki-1990}.

\begin{defn}
\label{def:ls}
A partition of $\gauss{V}{k}{q}$ into $N$ subspace designs, each with the parameters $t$-$(v,k,\lambda)_q$, is called a \emph{large set} $\LS_q[N](t,k,v)$.
More precisely, it is a collection
\[
    \{(V,\mathcal{B}_1), (V,\mathcal{B}_2), \ldots, (V,\mathcal{B}_N)\}
\]
of $t$-$(v,k,\lambda)_q$ designs such that $\{\mathcal{B}_1, \mathcal{B}_2,\ldots, \mathcal{B}_N\}$ is a partition of $\gauss{V}{k}{q}$.
In the case $N = 2$, the large set is also called a \emph{halving}.
\end{defn}

\begin{rem}
\label{rem:ls}
\item
\begin{enumerate}[(i)]
\item\label{rem:ls:lambda}
Note that the parameter $\lambda$ does not appear in the parameter set $\LS_q[N](t,k,v)$ of a large set.
This is because under the definition of a large set, $\lambda = \gauss{v-t}{k-t}{q} / N$ is already determined by the other parameters.
\item 
Large sets with $N = 1$ are called \emph{trivial}.
For all integers $0 \leq t \leq k \leq v$, the unique $\LS[1]_q(t,k,v)$ is given by $\left(V,\{\gauss{V}{k}{q}\}\right)$. 
\item\label{rem:ls:halving} For every $t$-$(v,k,\lambda)_q$ subspace design $(V,\mathcal{B})$, the \emph{supplementary} design $(V,\gauss{V}{k}{q}\setminus\mathcal{B})$ is again a subspace design with the parameters $t$-$(v,k,\gauss{v-t}{k-t}{q} - \lambda)_q$.
So in the case $\lambda = \gauss{v-t}{k-t}{q} / 2$, $\{ (V,\mathcal{B}), (V,\gauss{V}{k}{q}\setminus\mathcal{B})\}$ is a halving, showing that $t$-$(v,k,\gauss{v-t}{k-t}{q} / 2)_q$ subspace designs and halvings $\LS_q[2](t,k,v)$ are \enquote{the same.}
\end{enumerate}
\end{rem}

\begin{lem}
	\label{lem:ls_integrality}
	If there exists an $\LS_q[N](t,k,v)$, then for all $i\in\{0,\ldots,t\}$
	\[
		N\mid \gauss{v-i}{k-i}{q}\text{.}
	\]
\end{lem}

\begin{proof}
	An $\LS_q[N](t,k,v)$ consists of $t$-$(v,k,\lambda)_q$ subspace designs with $\lambda = \gauss{v-t}{k-t}{q} / N$.
	By \cite[Lemma~4.1(1)]{Suzuki-1990}, the numbers $\lambda\gauss{v-i}{k-i}{q} / \gauss{v-t}{k-t}{q}$ must be integers.
\end{proof}

In the case that the conditions of Lemma~\ref{lem:ls_integrality} are met, the parameter set $\LS_q[N](t,k,v)$ is called \emph{admissible}.
If an $\LS_q[N](t,k,v)$ in fact exists, the parameter set $\LS_q[N](t,k,v)$ is called \emph{realizable}.
By Lemma~\ref{lem:ls_integrality}, realizability implies admissibility.
For $t = 0$, also the converse is true:

\begin{lem}
	The large set parameters $\LS_q[N](0,k,v)$ are realizable if and only if they are admissible.
\end{lem}

\begin{proof}
	The design property for $t=0$ just means that all designs in the large set are of the same size.
	So the large set exists if and only if the total number $\gauss{v}{k}{q}$ of $k$-subsets is divisible by $N$, which is the condition in Lemma~\ref{lem:ls_integrality}.
\end{proof}

In the classical case $q = 1$, the above Lemma is still true for $t = 1$ \cite{Baranyai}, meaning that an $\LS[N](1,k,v)$ exists if and only if the parameters are admissible.
It is an open problem whether this result carries over to the $q$-analog case.

For checking the divisibility of Gaussian binomial coefficients, the generalization of Kummer's Theorem in~\cite{Fray,Knuth-Wilf-1989-JRAM396:212-219} is useful.
As a result, for fixed $N$, $t$, and $q$, the set of all $v$ and $k$ such that $\LS_q[N](t,k,v)$ is admissible carries kind of a fractal structure, see e.g.\ Table~\ref{tbl:ueberblick}.
A detailed discussion of this phenomenon for ordinary binomial coefficients and $N = 2$ can be found in~\cite{Wolfram}.

\begin{lem}
	\label{lem:ls_divisorsN}
	If $\LS_q[N](t,k,v)$ is realizable, then for each divisor $d\mid N$, $\LS_q[d](t,k,v)$ is realizable, too.
\end{lem}

\begin{lem}[{{\cite[Cor.~19]{trivial}}}]
\label{lem:resLS}
If there exists an $\LS_q[N](t,k,v)$ for $t\ge 1$ then there exists
\begin{enumerate}[(i)]
\item the \emph{dual} large set with parameters $\LS_q[N](t,v-k,v)$;
\item the \emph{reduced} large set with parameters $\LS_q[N](t-1,k,v)$;
\item a \emph{derived} large set with parameters $\LS_q[N](t-1,k-1,v-1)$;
\item a \emph{residual} large set with parameters $\LS_q[N](t-1,k,v-1)$.
\end{enumerate}
\end{lem}

%% file: decompositions.tex
\section{Decompositions of $\gauss{V}{k}{q}$ into joins}
\label{sect:decompositions}

\subsection{Joins of subspaces}
\label{subsect:join}

\begin{defn}
\label{defn:cover_avoid}
Let $U_1 \leq U_2 \leq V$ be a chain of subspaces of $V$.
We say that a subspace $K \leq V$
\begin{enumerate}[(i)]
	\item \emph{covers} the factor space $F = U_2/U_1$ if $U_1 + K = U_2 + K$,
	\item \emph{avoids} the factor space $F = U_2/U_1$ if $U_1 \cap K = U_2 \cap K$.
\end{enumerate}
\end{defn}

\begin{rem}
\item
\begin{enumerate}[(i)]
	\label{rem:cover_avoid}
	\item\label{rem:cover_avoid:history} Definition~\ref{defn:cover_avoid} is inspired by the corresponding notions in group theory, which had been introduced in \cite{Gaschuetz-1962} in the context of chief factors of finite solvable groups.
	\item\label{rem:cover_avoid:generalization} Principally, Definition~\ref{defn:cover_avoid} can be applied to any lattice.
	In the case $q=1$, i.e. the subset lattice $\mathcal{L}(X)$ of a set $X$, it is easy to see that $K\in\mathcal{L}(X)$ covers the flag $U_1\subseteq U_2$ if and only if $U_2\setminus U_1 \subseteq K$, and that $K$ avoids the flag $U_1\subseteq U_2$ if and only if $(U_2\setminus U_1) \cap K = \emptyset$.
	Back in the case $q\geq 2$, the following Lemma gives a similar description for factors of the form $V_j/V_i$ based on canonical matrices.
\end{enumerate}
\end{rem}

\begin{lem}
\label{lem:cover_avoid_matrices}
Let $U \leq \GF(q)^v$ and $i,j\in\{0,\ldots,v\}$ with $i\leq j$.
\begin{enumerate}[(a)]
	\item $U$ covers $V_j/V_i \iff I_j\setminus I_i \subseteq \pi(\cm(U))$.
	\item $U$ avoids $V_j/V_i \iff (I_j\setminus I_i) \cap \pi(\cm(U)) = \emptyset$.
\end{enumerate}
\end{lem}

\begin{proof}
Use Lemma~\ref{lem:Vi_property}.
\end{proof}

\begin{defn}
\label{defn:join}
\item
\begin{enumerate}[(a)]
\item\label{defn:join:ordinary_join} Let $U\leq V$ and $K_1,K_2 \leq V$ with $K_1\leq U\leq K_2$.
The \emph{(ordinary) join} of $K_1$ and $K_2/U$ with respect to $U$ is defined as
	\[
		K_1 \ast_U K_2/U = \{K \in\mathcal{L}(V) \mid U\cap K = K_1, U + K = K_2\}\text{.}
	\]
\item\label{defn:join:extended_join} Let $U_1 \leq U_2 \leq V$ and $K_1, K_2\leq V$ with $K_1\leq U_1$ and $U_2\leq K_2$.
We define the \emph{covering join} of $K_1$ and $K_2/U_2$ with respect to the factor space $F = U_2/U_1$ as
\[
	K_1 \ast_F K_2/U_2
	= \{K \in\mathcal{L}(V) \mid U_1 \cap K = K_1, U_2 + K = K_2, K\text{ covers } F\}
\]
and the \emph{avoiding join} of $K_1$ and $K_2$ with respect to the factor space $F = U_2/U_1$ as
\[
	K_1 \ast_{\bar{F}} K_2/U_2
	= \{K \in\mathcal{L}(V) \mid U_1 \cap K = K_1, U_2 + K = K_2, K\text{ avoids } F\}\text{.}
\]
\end{enumerate}
\end{defn}

The above join operators generalize the dot symbol for ordinary designs used in \cite{Ajoodani91}.

Figure~\ref{fig:ordinary_join} shows the Hasse diagram for a block $K$ of the ordinary join $K_1 \ast_U K_2/U$.
In Figure~\ref{fig:covering_avoiding_joins}, the Hasse diagram for a block $K$ of the covering join $K_1 \ast_{U_2/U_1} K_2/U_2$ and a block $\overline{K}$ of the avoiding join $K_1 \ast_{\overline{U_2/U_1}} K_2/U_2$ is shown.
The edge labels denote codimensions, where the symbols are defined as in Lemma~\ref{lem:join_basic_properties}.

\begin{figure}
\noindent\centering
\begin{tikzpicture}[scale=.8,auto]
    \node (O) at (0,0) {$\{\mathbf{0}\}$};
    \node (U) at (0,3.5) {$U$};
    \node (V) at (0,6) {$V$};
    \node (K1) at (2,1.5) {$K_1$};
    \node (K2) at (3.5,5.5) {$K_2$};
    \node (K) at (5.5,3.5) {$K$};
    \draw (O) -- node[font=\scriptsize,swap]{$k_1$} (K1) -- (U) -- node[font=\scriptsize,swap]{$\bar{k}_2 = k_2-u$} (K2) -- (V);
    \draw (K1) -- (K) -- (K2);
    \draw[loosely dotted] (O) -- node[font=\scriptsize]{$u$} (U) -- node[font=\scriptsize]{$v-u$} (V);
\end{tikzpicture}
\caption{Ordinary join of $K_1$ and $K_2/U$}
\label{fig:ordinary_join}
\end{figure}
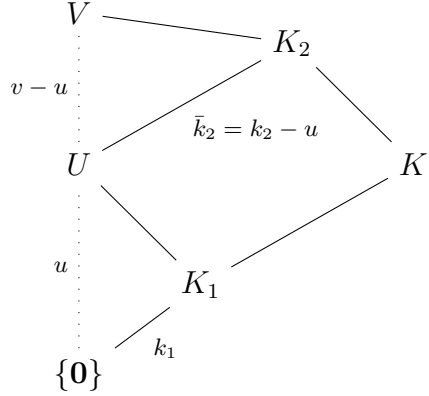

\begin{figure}
\noindent\centering
\begin{tikzpicture}[auto,scale=.8]
    \node (O) at (0,0) {$\{\mathbf{0}\}$};
    \node (U1) at (0,4) {$U_1$};
    \node (U2) at (0,6) {$U_2$};
    \node (V) at (0,8.5) {$V$};
    \node (K1) at (2,2) {$K_1$};
    \node (K2) at (3,7) {$K_2$};
    \coordinate (bot) at (2,4) {};
    \coordinate (top) at (3,5) {};
    \node (Kcov) at (5,5) {$K$};
    \node (Kav) at (5,3) {$\overline{K}$};
    \draw[fill] (bot) circle (2pt);
    \draw[fill] (top) circle (2pt);
    \draw (O) -- node[font=\scriptsize,swap]{$k_1$} (K1) -- (U1) -- node[font=\scriptsize]{$f = u_2-u_1$} (U2) -- node[font=\scriptsize]{$\bar{k}_2$}  (K2) -- (V);
    \draw (K1) -- (Kav) -- (top) -- (U1);
    \draw (K2) -- (Kcov) -- (bot) -- (U2);
    \draw (K1) -- (bot);
    \draw (K2) -- (top);
    \draw (Kav) -- (Kcov);
    \draw[loosely dotted] (O) -- node[font=\scriptsize]{$u_1$} (U1);
    \draw[loosely dotted] (U2) --node[font=\scriptsize]{$v-u_2$}  (V);
\end{tikzpicture}
\caption{Covering and avoiding join of $K_1$ and $K_2/U_2$}
\label{fig:covering_avoiding_joins}
\end{figure}
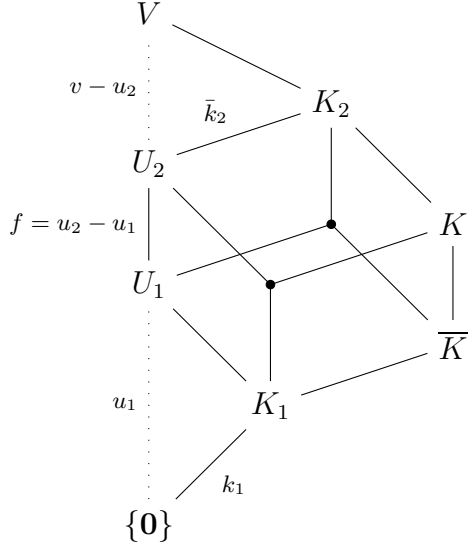

As a direct consequence of the definition, we get:
\begin{lem}
\label{lem:all_joins_the_same}
Fix the notation as in Definition~\ref{defn:join}.
Then
\begin{enumerate}[(a)]
\item\label{lem:all_joins_the_same:ord} $K_1 \ast_U K_2/U = K_1 \ast_{U/U} K_2/U = K_1\ast_{\overline{U/U}} K_2/U$,
\item\label{lem:all_joins_the_same:cov2ord} $K_1 \ast_{U_2/U_1} K_2/U_2 = K_1\ast_{U_1} K_2/U_1$,
\item\label{lem:all_joins_the_same:av2ord} $K_1 \ast_{\overline{U_2/U_1}} K_2/U_2 = K_1\ast_{U_2} K_2/U_2$.
\end{enumerate}
\end{lem}

\begin{rem}
By Lemma~\ref{lem:all_joins_the_same}\ref{lem:all_joins_the_same:cov2ord} and \ref{lem:all_joins_the_same:av2ord}, the covering and the avoiding join are a special case of the ordinary join.
So principally, everything what follows could be expressed in terms of the ordinary join only.
However, our main Theorem~\ref{thm:two_parameter_recursion} is based on the decomposition Theorems~\ref{thm:avoidpartition} and \ref{thm:coverpartition}, whose natural formulation relies on the avoiding join and the covering join, respectively.
A reformulation using the ordinary join would complicate the presentation and obscure the idea behind.
For that reason, we will develop the theory for all three kinds of the join.
Still, Lemma~\ref{lem:all_joins_the_same} allows us to shorten some proofs by reducing statements for the covering and the avoiding to the version for the ordinary join.
\end{rem}

\begin{lem}
\label{lem:join_basic_properties}
Fix the notation as in Definition~\ref{defn:join} and let $k_1 = \dim(K_1)$, $k_2 = \dim(K_2)$.
\begin{enumerate}[(a)]
	\item\label{lem:join_basic_properties:ordinary} For $u = \dim(U)$ and $\bar{k}_2 = \dim(K_2 / U) = k_2 - u$, we have
		\begin{align*}
	K_1\ast_U K_2/U & \subseteq \gauss{V}{k_1 + k_2 - u}{q} = \gauss{V}{k_1 + \bar{k}_2}{q}\text{,} \\
	\#(K_1\ast_U K_2/U) & = q^{(u - k_1) (k_2 - u)} = q^{(u - k_1) \bar{k}_2}\text{.}
		\end{align*}
	\item\label{lem:join_basic_properties:cover_avoid} For $u_1 = \dim(U_1)$, $u_2 = \dim(U_2)$, $f = \dim(F) = u_2 - u_1$ and $\bar{k}_2 = \dim(K_2 / U_2) = k_2 - u_2$, we have
		\begin{align*}
		K_1\ast_F K_2/U_2 & \subseteq \gauss{V}{k_1 + k_2 - u_1}{q} = \gauss{V}{k_1 + \bar{k}_2 + f}{q}\text{,} \\
		\#(K_1\ast_F K_2/U_2) & = q^{(u_1 - k_1)(k_2 - u_1)} = q^{(u_1 - k_1)(\bar{k}_2 + f)}
		\end{align*}
		and
		\begin{align*}
			K_1\ast_{\bar{F}} K_2/U_2 & \subseteq \gauss{V}{k_1 + k_2 - u_2}{q} = \gauss{V}{k_1 + \bar{k}_2}{q}\text{,} \\
			\#(K_1\ast_{\bar{F}} K_2/U_2) & = q^{(u_2 - k_1)(k_2 - u_2)} = q^{(u_2 - k_1)\bar{k}_2}\text{.}
		\end{align*}
\end{enumerate}
\end{lem}

\begin{proof}
By Lemma~\ref{lem:all_joins_the_same} it is enough to show the claim for the ordinary join.
For any subspace $K\in K_1\ast_U K_2/U$, the dimension formula yields $\dim(K) = \dim(U\cap K) + \dim(U + K) - \dim(U) = k_1 + k_2 - u$.
The join $K_1 \ast_U K_2/U$ consists exactly of those subspaces $K \leq V$ with $K_1\leq K$ such that $K/K_1$ is a complement of $U / K_1$ in $K_2 / K_1$, so its size is $q^{(u - k_1)((k_2 - k_1) - (u - k_1))} = q^{(u - k_1)(k_2 - u)}$.
\end{proof}

The concepts of cover and avoid are dual to each other:
\begin{lem}
	\label{lem:cover_avoid_duality}
	Let $U_1 \leq U_2 \leq V$.
	\begin{enumerate}[(a)]
		\item\label{lem:cover_avoid_duality:cover_avoid}
		Let $K \leq V$.
		\begin{enumerate}[(i)]
		    \item $K$ covers $U_2/U_1$ if and only if $K^\perp$ avoids $U_1^\perp/U_2^\perp$.
		    \item $K$ avoids $U_2/U_1$ if and only if $K^\perp$ covers $U_1^\perp/U_2^\perp$.
                \end{enumerate}
		\item\label{lem:cover_avoid_duality:join}
		Let $K_1 \leq U_1$ and $K_2 \leq V$ with $U_2 \leq K_2$.
		\begin{enumerate}[(i)]
		    \item $(K_1\ast_{U_2/U_1} K_2/U_2)^\perp = K_2^\perp\ast_{\overline{U_1^\perp/U_2^\perp}} K_1^\perp/U_1^\perp$.
		    \item $(K_1\ast_{\overline{U_2/U_1}} K_2/U_2)^\perp = K_2^\perp\ast_{U_1^\perp/U_2^\perp} K_1^\perp/U_1^\perp$.
                \end{enumerate}
	\end{enumerate}
\end{lem}

\begin{proof}
	From Definition~\ref{defn:join}, making use of $(A + B)^\perp = A^\perp \cap B^\perp$ and $(A \cap B)^\perp = A^\perp + B^\perp$ where $A,B\in\mathcal{L}(V)$.
\end{proof}

\subsection{Paths in the grid graph}
The directed \emph{grid graph} is defined as the vertex set $\N\times \N$ together with the set of directed edges
\[
    \{((x,y),(x+1,y)) \mid x,y\in\N\}\cup\{((x,y),(x,y+1)) \mid x,y\in\N\}\text{.}
\]
It is well known that the paths from $(0,0)$ to $(v-k,k)$ correspond bijectively to the $k$-subsets of $\{1,\ldots,v\}$.
For $K\in\binom{\{1,\ldots,v\}}{k}$, the corresponding path is constructed as follows: If $i\in K$ ($i\in\{1,\ldots,v\}$), then the $i$th step in the path is vertical direction, otherwise in horizontal direction.

For a $q$-analog of this property, we define the directed \emph{$q$-grid graph} in the same way, but with the difference that the horizontal edges $((x,y), (x+1,y))$ are assigned the multiplicity $q^y$ and labelled with the elements of $\GF(q)^y$.
Now a path starting in $(0,0)$ can be read as a column-wise description of a matrix in reduced row echelon form:
A vertical step corresponds to a pivot column, and a horizontal step corresponds to a non-pivot column having the entries given by the assigned label.
In this way, we get a one-to-one correspondence between the paths from $(0,0)$ to $(v-k,k)$ in the $q$-grid graph and the $(k\times v)$-matrices in reduced row echelon form.
So we have:

\begin{thm}\label{thm:grid_paths}
The above correspondence provides a bijection between the paths from $(0,0)$ to $(v-k,k)$ in the $q$-grid graph and the Graßmannian $\gauss{\GF(q)^v}{k}{q}$.
In particular, the number of paths from $(0,0)$ to $(v-k,k)$ in the $q$-grid graph is given by $\gauss{v}{k}{q}$.
\end{thm}

\begin{rem}
\item
\begin{enumerate}[(i)]
\item The shape of the path corresponding to a subspace $U \leq \GF(q)^v$ can be characterized without using its canonical matrix:
By Lemma~\ref{lem:cover_avoid_matrices}, the $i$-th step ($i\in\{1,\ldots,v\}$) is in vertical direction if and only if $U$ covers $V_{v-i+1}/V_{v-i}$, and it is in horizontal direction if and only if $U$ avoids $V_{v-i+1}/V_{v-i}$.
By Remark~\ref{rem:cover_avoid}\ref{rem:cover_avoid:generalization}, this cover-avoid-description carries over to the case $q = 1$, i.e. the representations of subsets by paths in the ordinary grid graph.
\item In Lemma~\ref{lem:cm_dual}, we have seen that for canonical matrices of dual subspaces, it is natural to switch to the reduced right row echelon form.
To get back to the reduced left row echelon form, we may reverse the coordinates of $\GF(q)^v$ afterwards.

So let $\rho : \GF(q)^v\to \GF(q)^v$, $(x_1,\ldots,x_v) \mapsto (x_v,\ldots,x_1)$ be the reversion automorphism.
Then $\rho(V_i^\perp) = V_{v-i}$ for all $i\in\{0,\ldots,v\}$, and if $P \subset \N\times\N$ is the path corresponding to $U \leq V$, by Lemma~\ref{lem:cover_avoid_matrices} the path corresponding to $\rho(U^\perp)$ is given by $\{(k-y,v-k-x) \mid (x,y)\in P\}$.
\item A variant of Theorem~\ref{thm:grid_paths} is found in \cite{Azose}.
However, the bijection between the Graßmannian and the set of paths is less explicit:
Instead of using multi-edges, the considered paths are weighted by the area below the path.
\end{enumerate}
\end{rem}

The subspaces $V_i$ admit a representation of all types of joins in terms of canonical matrices:

\begin{lem}
Let $0 \leq u_1 \leq u_2 \leq v$ be integers, $F = V_{u_2}/V_{u_1}$ and $f = \dim(F) = u_2 - u_1$.
Let $K_1,K_2 \leq \GF(q)^v$ with $K_1\leq V_{u_1}$ and $V_{u_2}\leq K_2$.
Furthermore, let $k_1 = \dim(K_1)$ and $\bar{k}_2 = \dim(K_2/V_{u_2})$.
The canonical matrices of $K_1$ and $K_2$ have the form
\[
	\begin{pmatrix}
		\mathbf{0}_{k_1 \times (v - u_1)} & A_1 \\
	\end{pmatrix}
	\qquad\text{and}\qquad
	\begin{pmatrix}
		A_2 & \mathbf{0}_{\bar{k}_2\times u_2} \\
		\mathbf{0}_{u_2\times (v-u_2)} & E_{u_2}
	\end{pmatrix}
\]
with canonical matrices $A_1\in \GF(q)^{k_1\times u_1}$ and $A_2\in \GF(q)^{\bar{k}_2\times (v - u_2)}$.
\begin{enumerate}[(a)]
	\item In the case $u_1 = u_2 = u$, the ordinary join $K_1 *_{V_u} (K_2/V_u)$ is given by all the subspaces of $\GF(q)^v$ with canonical matrices of the form
	\[
		\begin{pmatrix}
			A_2 & B \\
			\mathbf{0}_{k_1 \times (v - u)} & A_1 \\
		\end{pmatrix}
	\]
	where $B\in\GF(q)^{\bar{k}_2\times u}$ is a matrix having zero entries at all the pivot column positions of $A_1$.
	\item The covering join $K_1 *_F (K_2/V_{u_2})$ is given by all the subspaces of $\GF(q)^v$ with canonical matrices of the form
	\[
		\begin{pmatrix}
			A_2 & \mathbf{0}_{\bar{k}_2 \times f} & B_1 \\
			\mathbf{0}_{f \times (v - u_2)} & E_{f} & B_2 \\
			\mathbf{0}_{k_1\times (v - u_2)} & \mathbf{0}_{k_1\times f} & A_1
		\end{pmatrix}
	\]
	where $B_1\in\GF(q)^{\bar{k}_2 \times u_1}$, $B_2\in\GF(q)^{f\times u_1}$ are matrices having zero entries at all the pivot column positions of $A_1$.
	\item The avoiding join $K_1 *_{\bar{F}} (K_2/V_{u_2})$ is given by all the subspaces of $\GF(q)^v$ with canonical matrices of the form
	\[
		\begin{pmatrix}
			A_2 & B_1 & B_2 \\
			\mathbf{0}_{k_1 \times (v - u_2)} & \mathbf{0}_{k_1\times f} & A_1 \\
		\end{pmatrix}
	\]
	where $B_1\in\GF(q)^{\bar{k}_2 \times f}$ is an arbitrary matrix and $B_2\in\GF(q)^{\bar{k}_2\times u_1}$ is a matrix having zero entries at all the pivot column positions of $A_1$.
\end{enumerate}
\end{lem}

\begin{proof}
	Use Lemma~\ref{lem:Vi_property}.
\end{proof}

For all three types of join operators, we will extend the notation to sets of subspaces:
\begin{defn}
\label{def:extended_join}
Let $U_1 \leq U_2$ be subspaces of $V$, $k_1\in\{0,\ldots,\dim(U_1)\}$, $k_2\in\{\dim(U_2),\ldots,v\}$ and $\bar{k}_2 = k_2 - \dim(U_2)$.
Let $\mathcal{B}^{(1)} \subseteq \gauss{U_1}{k_1}{q}$ and $\mathcal{B}^{(2)} \subseteq \gauss{V/U_2}{\bar{k}_2}{q}$.
Furthermore, let $*$ denote the ordinary join $*_U$ (with $U = U_1 = U_2$) or the covering join $*_{U_2/U_1}$ or the avoiding join $*_{\overline{U_2/U_1}}$.
We define
\[
    \mathcal{B}^{(1)} \ast \mathcal{B}^{(2)}
    = \bigcup_{\substack{B^{(1)}\in\mathcal{B}^{(1)} \\ B^{(2)}\in\mathcal{B}^{(2)}}} B^{(1)} \ast B^{(2)}\text{.}
\]
Furthermore, we explicitly set the boundary cases
\[
    \mathcal{B}^{(1)} \ast \emptyset
    = \emptyset \ast \mathcal{B}^{(2)}
    = \emptyset \ast \emptyset
    = \emptyset\text{.}
\]
\end{defn}

\begin{cor}
\label{cor:join_to_path}
Let $u_1,u_2,k_1,\bar{k}_2$ be integers with $0 \leq k_1\leq u_1 \leq u_2 \leq \bar{k}_2 + u_2 \leq v$.
Furthermore, let $F = V_{u_2}/V_{u_1}$ and $f = \dim(F) = u_2 - u_1$.
\begin{enumerate}[(a)]
	\item\label{cor:join_to_path:ordinary} In the case $U_1 = U_2 = U$ and $u_1 = u_2 = u$, the ordinary join
\[
	\gauss{V_u}{k_1}{q} *_{V_u} \gauss{\GF(q)^v/V_u}{\bar{k}_2}{q}
	= \left\{K\in\gauss{\GF(q)^v}{k_1 + \bar{k}_2}{q} \mid \dim(K \cap V_u) = k_1\right\}
\]
is given by all subspaces of $\GF(q)^v$ whose representation in the $q$-grid graph is a path from $(0,0)$ to $(v-k_1-\bar{k}_2,k_1+\bar{k}_2)$ passing through the vertex $(v-u-\bar{k}_2,\bar{k}_2)$.
	\item\label{cor:join_to_path:cover} The covering join
	\[
	    \gauss{V_{u_1}}{k_1}{q} *_F \gauss{\GF(q)^v/V_{u_2}}{\bar{k}_2}{q}
	\]
	is given by all subspaces of $\GF(q)^v$ whose representation in the $q$-grid graph is a path from $(0,0)$ to $(v-k_1-\bar{k}_2-f,k_1+\bar{k}_2+f)$ passing through the vertical line segment from $(v-u_2-\bar{k}_2,\bar{k}_2)$ to $(v-u_2-\bar{k}_2,\bar{k}_2+f)$ of length $f$.
	\item\label{cor:join_to_path:avoid} The avoiding join
	\[
	    \gauss{V_{u_1}}{k_1}{q} *_{\bar{F}} \gauss{\GF(q)^v/V_{u_2}}{\bar{k}_2}{q}
	\]
	is given by all subspaces of $\GF(q)^v$ whose representation in the $q$-grid graph is a path from $(0,0)$ to $(v-k_1-\bar{k}_2,k_1+\bar{k}_2)$ passing through the horizontal line segment from $(v-u_2-\bar{k}_2,\bar{k}_2)$ to $(v-u_1-\bar{k}_2,\bar{k}_2)$ of length $f$.
\end{enumerate}
\end{cor}

From the definitions, it is straightforward to check
\begin{lem}
	\label{lem:join_monotonicity}
	Fix the notation as in Definition~\ref{def:extended_join} and let $\mathcal{B}_1^{(1)},\mathcal{B}_2^{(1)}$ be subsets of $\gauss{U_1}{k_1}{q}$ and $\mathcal{B}_1^{(2)},\mathcal{B}_2^{(2)}$ subsets of $\gauss{V/U_2}{k_2}{q}$.
	Then
	\[
	    \mathcal{B}_1^{(1)} * \mathcal{B}_1^{(2)} \subseteq \mathcal{B}_2^{(1)} * \mathcal{B}_2^{(2)}\iff\mathcal{B}_1^{(1)} \subseteq \mathcal{B}_2^{(1)}\text{ and }\mathcal{B}_1^{(2)} \subseteq \mathcal{B}_2^{(2)}\text{.}
	\]
\end{lem}

\subsection{Decompositions}
By the correspondence in Theorem~\ref{thm:grid_paths}, any partition of the set of paths from $(0,0)$ to $(v-k,k)$ in the $q$-grid graph yields a partition of the Graßmannian $\gauss{V}{k}{q}$.
Counting the sizes of the involved parts, each such partition yields a bijective proof for an identity for Gaussian binomial coefficients.

For our purpose, we are looking for decompositions of $\gauss{V}{k}{q}$ into joins in the sense of Definition~\ref{def:extended_join}.
By Corollary~\ref{cor:join_to_path}, this is the same as partitioning the set of paths from $(0,0)$ to $(v-k,k)$ into parts that are given by all paths through the same vertex or the same horizontal or vertical line segment.

To illustrate this approach, we look at the simplest nontrivial decomposition, which leads to a bijective proof of one of the well-known $q$-Pascal triangle identities.
\begin{lem}
	\label{lem:q_pascal}
	Let $1\leq k\leq v-1$ be integers.
	\begin{enumerate}[(a)]
		\item\label{lem:q_pascal:decomposition} A partition of $\gauss{V}{k}{q}$ is given by
		\[
			\gauss{V_{v-1}}{k-1}{q} *_{V/V_{v-1}} (\mathbf{0} + V)/V
			\quad\cup\quad
			\gauss{V_{v-1}}{k}{q} *_{\overline{V/V_{v-1}}} (\mathbf{0} + V)/V\text{.}
		\]
		\item\label{lem:q_pascal:identity} $\gauss{v}{k}{q} = q^{k-v} \gauss{v-1}{k-1}{q} + \gauss{v-1}{k}{q}$.
	\end{enumerate}
\end{lem}

\begin{proof}
	The first step of any path from $(0,0)$ to $(v-k,k)$ in the $q$-grid graph is either vertical or horizontal.
	By Corollary~\ref{cor:join_to_path}\ref{cor:join_to_path:cover} with $k_1 = k-1$, $\bar{k}_2 = 0$, $u_1 = v-1$, $u_2 = v$, the set of paths whose first step is vertical corresponds to the covering join $\gauss{V_{v-1}}{k-1}{q} *_{V/V_{v-1}} \gauss{V/V}{0}{q}$.
	In the same way, by Corollary~\ref{cor:join_to_path}\ref{cor:join_to_path:avoid} with $k_1 = k$, $\bar{k}_2 = 0$, $u_1 = v-1$, $u_2 = v$, the set of paths whose first step is horizontal corresponds to the avoiding join $\gauss{V_{v-1}}{k}{q} *_{\overline{V/V_{v-1}}} \gauss{V/V}{0}{q}$.
	This shows part~\ref{lem:q_pascal:decomposition}.
	From Lemma~\ref{lem:join_monotonicity} and Lemma~\ref{lem:join_basic_properties}\ref{lem:join_basic_properties:cover_avoid}, we get
	\begin{align*}
		\#\left(\gauss{V_{v-1}}{k-1}{q} *_{V/V_{v-1}} \gauss{V/V}{0}{q}\right) & = q^{v-k} \gauss{v-1}{k-1}{q}
		\quad\text{and} \\
		\#\left(\gauss{V_{v-1}}{k}{q} *_{\overline{V/V_{v-1}}} \gauss{V/V}{0}{q}\right) & = \gauss{v-1}{k}{q}
	\end{align*}
	and thus part~\ref{lem:q_pascal:identity}.
\end{proof}

In the following, we investigate decompositions where all involved joins are of the same kind.
While the decomposition into ordinary joins provides a bijective proof for the $q$-Vandermonde identity, the decompositions into avoiding or covering joins yield an apparently less well-known identity and will be used later in Section~\ref{sect:series} for the construction of infinite series of halvings.

\begin{thm}[Decomposition into ordinary joins] \label{thm:qVandermonde}
Let $U \leq V$ with $\dim(U) = u$ and $k\in\{0,\ldots,v\}$.
A partition of $\gauss{V}{k}{q}$ is obtained by
\[
	\gauss{V}{k}{q} = \bigcup_{i\in\Z} \left(\gauss{U}{i}{q} *_U \gauss{V/U}{k-i}{q}\right)
\]
The resulting counting formula is the \emph{$q$-Vandermonde identity}
\[
	\gauss{v}{k}{q} = \sum_{i\in\Z} q^{(u-i)(k-i)} \gauss{u}{i}{q} \cdot \gauss{v-u}{k-i}{q}\text{.}
\]
\end{thm}

\begin{proof}
	Without restriction, let $V = \GF(q)^v$ and $U = V_u$.
	The $k$-subspaces of $V$ correspond to the paths from $(0,0)$ to $(v-k,k)$ in the $q$-grid graph.
	For each such path there is a unique $i\in\{0,\ldots,k\}$ such that the vertex $(v-u-k+i,k-i)$ is on the path.
	This induces a partition on the set of paths.
	Now the application of Corollary~\ref{cor:join_to_path}\ref{cor:join_to_path:ordinary} with $k_1 = i$, $\bar{k}_2 = k-i$ and the same $u$ yields the claimed partition of the Graßmannian.
\end{proof}

\begin{rem}
	Of course, in Theorem~\ref{thm:qVandermonde} only finitely many sets in the union are non-empty and only finitely many terms in the sum are non-zero.
	More precisely, the non-vanishing expressions are those with $\max(0,k+u-v) \leq i \leq\min (u,k)$.
\end{rem}

\begin{figure}
\noindent\centering
\begin{tikzpicture}[scale=.8]
    \draw[step=1.0,black,thin] (0,0) grid (7,3);
    \node[left] at (0,0) {$(0,0)$};
    \node[right] at (7,3) {$(7,3)$};
    \draw[fill] (4,0) circle (4pt);
    \draw[fill] (3,1) circle (4pt);
    \draw[fill] (2,2) circle (4pt);
    \draw[fill] (1,3) circle (4pt);
\end{tikzpicture}
\caption{Decomposition of $\gauss{\GF(q)^{10}}{3}{q}$ into ordinary joins}
\label{fig:path_qvandermonde}
\end{figure}

\begin{ex}
For $v = 10$, $k = 3$ and $u = 6$ the corresponding partition of the paths is shown in Figure~\ref{fig:path_qvandermonde}.
The resulting identity is
\begin{align*}
\gauss{10}{3}{q} & = q^{18}\gauss{6}{0}{q} \gauss{4}{3}{q}
                   + q^{10} \gauss{6}{1}{q} \gauss{4}{2}{q}
                   + q^4 \gauss{6}{2}{q} \gauss{4}{1}{q} 
                   + q^0 \gauss{6}{3}{q} \gauss{4}{0}{q}\text{.}
\end{align*}
\end{ex}

\begin{thm}[Decomposition into avoiding joins]\label{thm:avoidpartition}
Let
\[
	\{\mathbf 0\} = U_0 < U_1<\ldots < U_v = V
\]
be a maximal chain of subspaces of $V$, $k\in\{0,\ldots,v\}$ and $s \in \{0, \ldots, v-k-1\}$.
A partition of $\gauss{V}{k}{q}$ is obtained by
\[
	\gauss{V}{k}{q}
	= \bigcup_{i=0}^k \gauss{U_{s+i}}{i}{q} *_{\overline{U_{s+i+1}/U_{s+i}}} \gauss{V/U_{s+i+1}}{k-i}{q}\text{.}
\]
It yields the identity
\[
	\gauss{v}{k}{q} = \sum_{i=0}^k q^{(s+1)(k-i)}\gauss{s+i}{i}{q}\cdot\gauss{v-s-i-1}{k-i}{q}\text{.}
\]
\end{thm}

\begin{proof}
	Without restriction, let $V = \GF(q)^v$ and $U_i = V_i$ as defined in Section~\ref{sect:prelim}.
	For each path from $(0,0)$ to $(v-k,k)$ in the $q$-grid graph, there is a unique $i\in\{0,\ldots,k\}$ such that the path is passing though the horizontal line segment from $(v-k-s-1,k-i)$ to $(v-k-s,k-i)$.
	The application of Corollary~\ref{cor:join_to_path}\ref{cor:join_to_path:avoid} with $k_1 = i$, $\bar{k}_2 = k-i$, $u_1 = s+i$ and $u_2 = s+i+1$ yields the claimed partition of the Graßmannian.
\end{proof}

\begin{figure}
\noindent\centering
\begin{tikzpicture}[scale=.8]
    \draw[step=1.0,black,thin] (0,0) grid (7,3);
    \node[left] at (0,0) {$(0,0)$};
    \node[right] at (7,3) {$(7,3)$};
    \draw[line width = 4pt] (3,0) -- (4,0);
    \draw[line width = 4pt] (3,1) -- (4,1);
    \draw[line width = 4pt] (3,2) -- (4,2);
    \draw[line width = 4pt] (3,3) -- (4,3);
    \draw[fill] (3,0) circle (2pt);
    \draw[fill] (3,1) circle (2pt);
    \draw[fill] (3,2) circle (2pt);
    \draw[fill] (3,3) circle (2pt);
    \draw[fill] (4,0) circle (2pt);
    \draw[fill] (4,1) circle (2pt);
    \draw[fill] (4,2) circle (2pt);
    \draw[fill] (4,3) circle (2pt);
\end{tikzpicture}
\caption{Decomposition of $\gauss{\GF(q)^{10}}{3}{q}$ into avoiding joins}
\label{fig:path_pure_avoid}
\end{figure}

\begin{ex}
\label{ex:decomposition_avoid_join}
For $v = 10$, $k = 3$ and $s = 3$ the corresponding partition of the paths is shown in Figure~\ref{fig:path_pure_avoid}.
The resulting identity is
\begin{align*}
\gauss{10}{3}{q} & = q^{12}\gauss{3}{0}{q}\gauss{6}{3}{q} + q^8 \gauss{4}{1}{q}\gauss{5}{2}{q}  + q^4 \gauss{5}{2}{q}\gauss{4}{1}{q} + q^0 \gauss{6}{3}{q}\gauss{3}{0}{q}\text{.}
\end{align*}
\end{ex}

\begin{thm}[Decomposition into covering joins]\label{thm:coverpartition}
Let
\[
	\{\mathbf 0\} = U_0 < U_1<\ldots < U_v = V
\]
be a maximal chain of subspaces of $V$, $k\in\{0,\ldots,v\}$ and $s \in \{0,\ldots,k-1\}$.
A partition of $\gauss{V}{k}{q}$ is obtained by
\[
	\gauss{V}{k}{q}
	= \bigcup_{i=0}^{v-k} \gauss{U_{v-s-i-1}}{k-s-1}{q} *_{U_{v-s-i}/U_{v-s-i-1}} \gauss{V/U_{v-s-i}}{s}{q}\text{.}
\]
It yields the identity
\[
	\gauss{v}{k}{q} = \sum_{i=0}^{v-k} q^{(v-k-i)(s+1)}\gauss{v-s-i-1}{k-s-1}{q}\cdot\gauss{s+i}{s}{q}\text{.}
\]
\end{thm}

\begin{proof}
For each path from $(0,0)$ to $(v-k,k)$ in the $q$-grid graph, there is a unique $i\in\{0,\ldots,v-k\}$ such that the path is passing though the vertical line segment from $(i,s)$ to $(i,s+1)$.
    The application of Corollary~\ref{cor:join_to_path}\ref{cor:join_to_path:cover} with $k_1 = k-s-1$, $\bar{k}_2 = s$, $u_1 = v-s-i-1$ and $u_2 = v-s-i$ yields the claimed partition of the Graßmannian.
\end{proof}

\begin{rem}
	\item
	\begin{enumerate}[(i)]
	\item Theorem~\ref{thm:coverpartition} is a dualized version of Theorem~\ref{thm:avoidpartition}.
	More precisely, applying Lemma~\ref{lem:cover_avoid_duality}\ref{lem:cover_avoid_duality:join}, Theorem~\ref{thm:coverpartition} arises from taking the duals in Theorem~\ref{thm:avoidpartition} with respect to the standard bilinear form, reversing the order of the coordinates and substituting $v-k$ by $k$.
	Consequently, the resulting counting formula (with $k$ set to $v-k$ in Theorem~\ref{thm:avoidpartition} and using $\gauss{a}{b}{q} = \gauss{a}{a-b}{q}$) is the same for both theorems.

	\item The counting formulas in Theorem~\ref{thm:qVandermonde} and~\ref{thm:avoidpartition} are special cases of the Theorem in~\cite{Bender} (with $a_i = s+i$ and $a_i = u$, respectively).
	The \enquote{proof by geometry} in~\cite{Bender} can be interpreted as a possibly mixed path decomposition into ordinary and avoiding joins.
	\end{enumerate}
\end{rem}

\begin{figure}
\noindent\centering
\begin{tikzpicture}[scale=.8]
    \draw[step=1.0,black,thin] (0,0) grid (7,3);
    \node[left] at (0,0) {$(0,0)$};
    \node[right] at (7,3) {$(7,3)$};
    \draw[line width = 4pt] (0,1) -- (0,2);
    \draw[line width = 4pt] (1,1) -- (1,2);
    \draw[line width = 4pt] (2,1) -- (2,2);
    \draw[line width = 4pt] (3,1) -- (3,2);
    \draw[line width = 4pt] (4,1) -- (4,2);
    \draw[line width = 4pt] (5,1) -- (5,2);
    \draw[line width = 4pt] (6,1) -- (6,2);
    \draw[line width = 4pt] (7,1) -- (7,2);
    \draw[fill] (0,1) circle (2pt);
    \draw[fill] (1,1) circle (2pt);
    \draw[fill] (2,1) circle (2pt);
    \draw[fill] (3,1) circle (2pt);
    \draw[fill] (4,1) circle (2pt);
    \draw[fill] (5,1) circle (2pt);
    \draw[fill] (6,1) circle (2pt);
    \draw[fill] (7,1) circle (2pt);
    \draw[fill] (0,2) circle (2pt);
    \draw[fill] (1,2) circle (2pt);
    \draw[fill] (2,2) circle (2pt);
    \draw[fill] (3,2) circle (2pt);
    \draw[fill] (4,2) circle (2pt);
    \draw[fill] (5,2) circle (2pt);
    \draw[fill] (6,2) circle (2pt);
    \draw[fill] (7,2) circle (2pt);
\end{tikzpicture}
\caption{Decomposition of $\gauss{\GF(q)^{10}}{3}{q}$ into covering joins}
\label{fig:path_pure_cover}
\end{figure}

\begin{ex}
For $v = 10$, $k = 3$ and $s = 1$ the corresponding partition of the paths is shown in Figure~\ref{fig:path_pure_cover}.
The resulting identity is
\begin{align*}
\gauss{10}{3}{q} & = 
q^{14}\gauss{8}{1}{q}\gauss{1}{1}{q}
+ q^{12}\gauss{7}{1}{q}\gauss{2}{1}{q}
+ q^{10}\gauss{6}{1}{q}\gauss{3}{1}{q}
+ q^{8}\gauss{5}{1}{q}\gauss{4}{1}{q} \\
& \phantom{{}={}}
+ q^{6}\gauss{4}{1}{q}\gauss{5}{1}{q}
+ q^{4}\gauss{3}{1}{q}\gauss{6}{1}{q}
+ q^{2}\gauss{2}{1}{q}\gauss{7}{1}{q}
+ q^{0}\gauss{1}{1}{q}\gauss{8}{1}{q}\text{.}
\end{align*}
\end{ex}

%% file: n-t-partitionable.tex
\section{$(N,t)$-partitionable sets}
\label{sect:n-t-partitionable}
The content of this section can be seen as a $q$-analog of parts of \cite{Ajoodani94}, where a similar theory is developed for the set case $q = 1$.

The \emph{zeta function} of a poset $(X,\leq)$ is defined as
\[
    \zeta : X\times X \to \Z\text{,}\quad (x,y) \mapsto \begin{cases}1 & \text{if }x\leq y\text{,}\\0 & \text{otherwise.}\end{cases}
\]
For the poset $(\mathcal{L}(V),\leq)$, we extend the zeta function to sets of subspaces by
\[
	\lambda : \mathcal{L}(V) \times \mathcal{P}(\mathcal{L}(V))\to\Z\text{,}\quad U \times\mathcal{B} \mapsto \sum_{B\in\mathcal{B}} \zeta(U,B) = \#\{B\in\mathcal{B} \mid U \leq B\}\text{.}
\]

\begin{defn}
\label{defn:t_equiv}
Let $t\in\{0,\ldots,v\}$ and $\mathcal{B}_1$ and $\mathcal{B}_2$ be two sets of $k$-subspaces of $V$.
$\mathcal{B}_1$ and $\mathcal{B}_2$ are called \emph{$t$-equivalent} if for all $T\in\gauss{V}{t}{q}$
\[
    \lambda(T,\mathcal{B}_1) = \lambda(T,\mathcal{B}_2)\text{.}
\]
\end{defn}

\begin{rem}
	\item
	\begin{enumerate}[(a)]
    		\item Note that in the above definition, the number $\lambda(T,\mathcal{B}_i)$ may differ for different choices of $T$.
		\item The property of being $t$-equivalent does not depend on the exact choice of the ambient space $V$.
		This will follow from Lemma~\ref{lem:teq_ambientspace}\ref{lem:teq_ambientspace:top}.
		\item Two sets of $k$-subspaces $\mathcal{B}_1$ and $\mathcal{B}_2$ are $0$-equivalent if and only if $\#\mathcal{B}_1 = \#\mathcal{B}_2$.
		\item In the literature, pairs $(\mathcal{B}_1,\mathcal{B}_2)$ of $t$-equivalent sets are also called \emph{trades} or \emph{bitrades}, see \cite{Khosrovshahi-TayfehRezaie-2009} for the situation of classical block designs.
		The minimum possible size of trades for subspace designs has been investigated recently in \cite{Krotov-1,Krotov-2}.
	\end{enumerate}
\end{rem}

\begin{lem}\label{tequiv}
Let $0\leq s\leq t \leq k \leq v$ be integers. 
\begin{enumerate}[(a)]
	\item\label{tequiv:a} Let $\mathcal{B} \subseteq \gauss{V}{k}{q}$ and $S\in\gauss{V}{s}{q}$.
	Then
	\[
	    \lambda(S,\mathcal{B}) = \Bigl(\sum_{S\leq \underaccent{\dot}{T}\in\gauss{V}{t}{q}} \lambda(T,\mathcal{B})\Bigr) / \gauss{k-s}{t-s}{q}\text{.}
	\]
	\item\label{tequiv:b}
	If $\mathcal{B}_1, \mathcal{B}_2\subseteq\gauss{V}{k}{q}$ are $t$-equivalent then they are also $s$-equivalent.
	In particular, $\mathcal{B}_1$ and $\mathcal{B}_2$ are $0$-equivalent, so $\#\mathcal{B}_1 = \#\mathcal{B}_2$.
	\end{enumerate}
\end{lem}

\begin{proof}
	For~\ref{tequiv:a}, count the set $\{(T,B) \mid T\in\gauss{V}{t}{q}, B\in\mathcal{B}, S\leq T\leq B\}$ in two ways.
	Part~\ref{tequiv:b} is a direct consequence.
\end{proof}

\begin{lem}
	\label{lem:teq_ambientspace}
	Let $U \leq V$, $k\in\{0,\ldots,v\}$ and $\mathcal{B}_1,\mathcal{B}_2 \subseteq \gauss{V}{k}{q}$.
	\begin{enumerate}[(a)]
		\item\label{lem:teq_ambientspace:top} If $B \leq U$ for all $B\in\mathcal{B}_1\cup\mathcal{B}_2$, then $\mathcal{B}_1$ and $\mathcal{B}_2$ are $t$-equivalent in $V$ if and only if they are $t$-equivalent in $U$.
		\item\label{lem:teq_ambientspace:bot} If $U \leq B$ for all $B\in\mathcal{B}_1\cup \mathcal{B}_2$, then $\mathcal{B}_1$ and $\mathcal{B}_2$ are $t$-equivalent in $V$ if and only if $\{B/U \mid B\in\mathcal{B}_1\}$ and $\{B/U \mid B\in\mathcal{B}_2\}$ are $t$-equivalent in $V/U$.
	\end{enumerate}
\end{lem}

\begin{proof}
	For part~\ref{lem:teq_ambientspace:top}, the \enquote{only if}-direction is trivial, and the \enquote{if}-direction follows from Lemma~\ref{tequiv}\ref{tequiv:b}.

	Part~\ref{lem:teq_ambientspace:bot} is done similarly.
	Since $U \leq B$ for all $B\in\mathcal{B}_1\cup \mathcal{B}_2$, we have $T \leq B$ if and only if $(T+U)/U \leq B/U$.
	This immediately gives the \enquote{if}-direction.
	The \enquote{only if}-direction follows from $\dim_{V/U} (T+U)/U \leq t$ for all $T\in\gauss{V}{t}{q}$ and Lemma~\ref{tequiv}\ref{tequiv:b}.
\end{proof}

\begin{defn}
Let $0\leq t\leq k\leq v$ be integers, $\mathcal{B}$ a set of $k$-subspaces of $V$ and $N$ a positive integer.
A partition $\{\mathcal{B}_1,\ldots,\mathcal{B}_N\}$ of $\mathcal{B}$ into $N$ parts is called an \emph{$(N,t)$-partition} if the parts $\mathcal{B}_i$ are pairwise $t$-equivalent.
The set $\mathcal{B}$ is called \emph{$(N,t)$-partitionable} if there exists an $(N,t)$-partition of $\mathcal{B}$.
Furthermore, we extend the notion $(N,t)$-partitionable to the value $t=-1$ by unconditionally calling any set of $k$-subspaces \emph{$(N,-1)$-partitionable}.
\end{defn}

\begin{rem}
By Lemma~\ref{tequiv}, if $\{\mathcal{B}_1,\ldots,\mathcal{B}_N\}$ is an $(N,t)$-partition of $\mathcal B$ with an integer $t\geq 0$, then in particular it is an $(N,0)$-partition of $\mathcal{B}$, showing that all parts $\mathcal{B}_i$ are of the same size.
\end{rem}

\begin{lem}
\label{lem:disjoint_union_partitionable}
Let $\mathcal{B}^{(1)}$ and $\mathcal{B}^{(2)}$ be two disjoint $(N,t)$-partitionable subsets of $\gauss{V}{k}{q}$ with an integer $t\geq -1$.
Then also $\mathcal{B}^{(1)}\cup \mathcal{B}^{(2)}$ is $(N,t)$-partitionable.
\end{lem}

\begin{proof}
	For the border case $t = -1$, there is nothing to show.
	For $t\geq 0$, let $\{\mathcal{B}_1^{(1)},\ldots,\mathcal{B}_N^{(1)}\}$ be an $(N,t)$-partition of $\mathcal{B}^{(1)}$ and  $\{\mathcal{B}_1^{(2)},\ldots,\mathcal{B}_N^{(2)}\}$ be an $(N,t)$-partition of $\mathcal{B}^{(2)}$.
	Then $\{\mathcal{B}_1^{(1)}\cup \mathcal{B}_1^{(2)},\ldots,\mathcal{B}_N^{(1)}\cup \mathcal{B}_N^{(2)}\}$ is an $(N,t)$-partition of $\mathcal{B}^{(1)}\cup \mathcal{B}^{(2)}$.
\end{proof}

\begin{lem}\label{lma:ls_to_partitionable}
	Let $\mathcal{B}_1,\ldots,\mathcal{B}_N \subseteq \gauss{V}{k}{q}$.
	Then $\{\mathcal{B}_1,\ldots,\mathcal{B}_N\}$ is an $(N,t)$-partition of $\gauss{V}{k}{q}$ if and only if $\{(V,\mathcal{B}_1),\ldots,(V,\mathcal{B}_N)\}$ is an $\LS_q[N](t,k,v)$.
\end{lem}

\begin{proof}
The direction \enquote{$\Leftarrow$} is clear.
For \enquote{$\Rightarrow$}, let $T\in\gauss{V}{t}{q}$.
Since all $k$-subsets of $V$ are covered by the $(N,t)$-partition, $\lambda(T,\mathcal B_1) + \ldots + \lambda(T,\mathcal B_N) = \gauss{v-t}{k-t}{q}$.
Furthermore $\lambda(T,\mathcal B_1) = \ldots = \lambda(T,\mathcal B_N)$ by the property of an $(N,t)$-partition.
So the number $\lambda(T,\mathcal B_i) = \gauss{v-t}{k-t}{q} / N$ does not depend on the choice of $T\in\gauss{V}{t}{q}$ for $i\in\{1,\ldots,N\}$, showing that each part $\mathcal{B}_i$ forms a $t$-$(v,k,\gauss{v-t}{k-t}{q}/N)_q$ subspace design.
\end{proof}

In the following, we provide tools to combine $(N,t)$-partitionable sets of subspaces of two independent vector spaces $V_1$ and $V_2$ over $\GF(q)$ to an $(N,t)$-partitionable set of subspaces of a suitable vector space $V$.

\begin{lem}
	\label{lem:lambda_join}
	Let $K_1 \leq U \leq K_2 \leq V$ be a chain of subspaces.
	For each subspace $T \leq V$,
	\begin{multline*}
	    \{ K \in K_1 *_U K_2/U \mid T \leq K \} \\
	    = \begin{cases} (K_1 + T) *_{U + T} K_2/(U+T) & \text{if }U\cap T\leq K_1\text{ and }T\leq K_2\text{,} \\ \emptyset & \text{otherwise.}\end{cases}
	\end{multline*}
	Setting $u = \dim(U)$, $k_1 = \dim(K_1)$, $\bar{k}_2 = \dim(K_2/U)$ and $r = \dim((U+T)/U)$,
	\[
		\lambda(T, K_1 *_U K_2/U)
		= \begin{cases}
		    q^{(u - k_1)(\bar{k}_2 - r)} & \text{if } U\cap T\leq K_1\text{ and }T\leq K_2\text{,} \\
		    0 & \text{otherwise.}
		\end{cases}
	\]
\end{lem}

\begin{proof}
	The situation with $U\cap T \leq K_1$ and $T\leq K_2$ is illustrated in figure~\ref{fig:lambda_join}.
	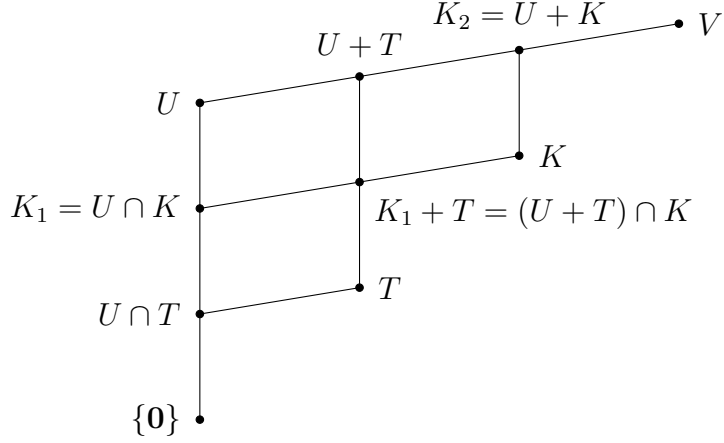
\begin{figure}
	\centering
	\begin{tikzpicture}[scale=.7]
	\coordinate [label={[label distance=3pt]180:$\{\mathbf{0}\}$}] (O) at (0,0);
	\coordinate [label={[label distance=3pt]180:$U\cap T$}] (UT) at (0,2);
	\coordinate [label={[label distance=3pt]180:$K_1 = U\cap K$}] (K1) at (0,4);
	\coordinate [label={[label distance=3pt]180:$U$}] (U) at (0,6);
	\coordinate [label={[label distance=3pt]0:$T$}] (T) at (3,2.5);
	\coordinate [label={[label distance=2pt]315:$K_1+T = (U+T)\cap K$}] (pivot) at (3,4.5);
	\coordinate [label={[label distance=3pt]90:$U+T$}] (UpT) at (3,6.5);
	\coordinate [label={[label distance=3pt]0:$K$}] (K) at (6,5);
	\coordinate [label={[label distance=4pt]90:$K_2=U+K$}] (K2) at (6,7);
	\coordinate [label={[label distance=3pt]0:$V$}](V) at (9,7.5);
	\draw (O) -- (UT) -- (K1) -- (U) -- (UpT) -- (K2) -- (V);
	\draw (UT) -- (T) -- (pivot) -- (UpT);
	\draw (K1) -- (pivot) -- (K) -- (K2);
	\draw [fill] (O) circle (2pt);
	\draw [fill] (UT) circle (2pt);
	\draw [fill] (K1) circle (2pt);
	\draw [fill] (U) circle (2pt);
	\draw [fill] (T) circle (2pt);
	\draw [fill] (pivot) circle (2pt);
	\draw [fill] (UpT) circle (2pt);
	\draw [fill] (K) circle (2pt);
	\draw [fill] (K2) circle (2pt);
	\draw [fill] (V) circle (2pt);
	\end{tikzpicture}
	\caption{Hasse diagram for the situation in Lemma~\ref{lem:lambda_join}}
	\label{fig:lambda_join}
	\end{figure}

	For \enquote{$\subseteq$}, assume there is a $K\in K_1 *_U K_2/U$ with $T \leq K$.
	Then $K\cap U = K_1$, $K + U = K_2$.
	So $T + K_2 = T + (K + U) = (T + K) + U = K + U = K_2$, showing $T \leq K_2$.
	Furthermore, $K_1\cap (U \cap T) = (K \cap U) \cap (U \cap T) = U \cap (K \cap T) = U \cap T$, showing that $U\cap T \leq K_1$.
	This already implies the empty set case.
	In the other case, $K + (U + T) = (K + U) + T = K_2 + T = K_2$.
	Furthermore, $T \leq K$ allows the application of the modularity law such that $K \cap (U + T) = (K \cap U) + T = K_1 + T$.
	So $K \in (K_1 + T)*_{U + T} K_2/(U+T)$.

	For \enquote{$\supseteq$}, let $U\cap T\leq K_1$ and $K \in (K_1 + T) *_{U+T} K_2/(U+T)$.
	Then $K \cap (U + T) = K_1 + T$ and $K + (U + T) = K_2$.
	So $T \leq K_1 + T \leq K$ and $K + U = (K + T) + U = K + (U + T) = K_2$.
	In addition, $K\cap U = (K\cap (U + T)) \cap U = (K_1 + T) \cap U = K_1 + (U \cap T) = K_1$, where the modularity law was used with $K_1 \leq U$.
	This shows $K \in K_1 *_U K_2/U$.

	Let $U\cap T\leq K_1$ and $T\leq K_2$.
	By $K_1 \leq U$ and the modularity law, $(K_1 + T)\cap U = K_1 + (T \cap U) = K_1$ and $K_1 + T + U = U + T$.
	Now the dimension formula yields
	\begin{align*}
		\dim(U + T) - \dim(K_1 + T)
		& = \dim((K_1 + T) + U) - \dim(K_1 + T) \\
		& = \dim(U) - \dim((K_1 + T)\cap U) \\
		& = \dim(U) - \dim(K_1) \\
		& = u - k_1\text{.}
	\end{align*}
	Furthermore, we have
	\[
		\dim(K_2) - \dim(U + T)
		= (\bar{k}_2 + u) - (r + u)
		= \bar{k}_2 - r\text{,}
	\]
	By Lemma~\ref{lem:join_basic_properties}\ref{lem:join_basic_properties:ordinary}, we get
	\begin{align*}
		\lambda(T, K_1 *_U K_2/U)
		& = \#((K_1 + T) *_{U + T} K_2/(U+T)) \\
		& = q^{(\dim(U+T) - \dim(K_1+T))(\dim(K_2) - \dim(U+T))} \\
		& = q^{(u-k_1)(\bar{k}_2 - r)}\text{.}
	\end{align*}
\end{proof}

\begin{lem}[{{Basic Lemma; $q$-analog of \cite[Lemma~1]{Ajoodani94}}}]
	\label{lem:basic}
	Let $U_1 \leq U_2 \leq V$, $k_1 \in\{0,\ldots,\dim(U_1)\}$, $\bar{k}_2\in\{0,\ldots,\dim(V/U_2)\}$ and $N$ a positive integer.
	Furthermore, let $*$ denote the ordinary join $*_U$ (with $U = U_1 = U_2$) or the covering join $*_{U_2/U_1}$ or the avoiding join $*_{\overline{U_2/U_1}}$.

	If $\mathcal{B}^{(1)} \subseteq \gauss{U}{k_1}{q}$ is $(N,t_1)$-partitionable and $\mathcal{B}^{(2)}\subseteq \gauss{U}{\bar{k}_2}{q}$ is $(N,t_2)$-partitionable with integers $t_1, t_2\geq -1$, then $\mathcal{B}^{(1)}\ast\mathcal{B}^{(2)}$ is $(N,t_1+t_2+1)$-partitionable.
\end{lem}

\begin{proof}
	By Lemma~\ref{lem:all_joins_the_same}\ref{lem:all_joins_the_same:cov2ord} and \ref{lem:all_joins_the_same:av2ord} and Lemma~\ref{lem:teq_ambientspace}, it is enough to consider the ordinary join $*_U$.

	Let $\dim(U) = u$.
	If $t_1 \geq 0$, an $(N,t_1)$-partition of $\mathcal{B}^{(1)}$ is denoted by $\{\mathcal{B}^{(1)}_1,\ldots,\mathcal{B}^{(1)}_N\}$, and if $t_2 \geq 0$, an $(N,t_2)$-partition of $\mathcal{B}^{(2)}$ is denoted by $\{\mathcal{B}^{(2)}_1,\ldots,\mathcal{B}^{(2)}_N\}$.

	For $t_1 = t_2 = -1$, there is nothing to show.

	Next, we consider the case $t_1 \geq 0$ and $t_2 = -1$.
	By Lemma~\ref{lem:join_monotonicity},
	\[
		\mathcal{S} = \{\mathcal{B}_1^{(1)}\ast_U \mathcal{B}^{(2)},\ldots ,\mathcal{B}_N^{(1)}\ast_U \mathcal{B}^{(2)}\}
	\]
	is a partition of $\mathcal{B}^{(1)} *_U \mathcal{B}^{(2)}$.
	Let $T\in\gauss{V}{t_1}{q}$ and $r = \dim((U+T)/U)$.
	Then by Lemma~\ref{lem:lambda_join}, for all $i\in\{1,\ldots,N\}$ we have
	\[
		\lambda(T,\mathcal{B}_i^{(1)}\ast_U \mathcal{B}^{(2)})
		= \lambda(U \cap T, \mathcal{B}_i^{(1)})\cdot \lambda((U+T)/U, \mathcal{B}^{(2)})\cdot q^{(u-k_1)(\bar{k}_2 - r)}\text{.}
	\]
	Since $\mathcal{B}^{(1)}$ is an $(N,t_1)$-partition and $\dim(U\cap T) \leq t_1$, this expression is independent of $i$ and therefore, $\mathcal{S}$ is indeed an $(N,t_1)$-partition.

	The case $t_1 = -1$ and $t_2 \geq 0$ is done similarly.%
	\footnote{In fact, $t_1 = -1$, $t_2\geq 0$ is the dual situation of the already considered $t_1 \geq 0$, $t_2 = -1$.}

	Now let $t_1 \geq 0$ and $t_2 \geq 0$.
	Let $A \in \{1,\ldots,N\}^{N\times N}$ be a Latin square of size $N\times N$.
	For $i\in\{1,\ldots,N\}$ we define
	\[
	    \mathcal{S}_i
	    = \bigcup\{\mathcal{B}^{(1)}_x *_U \mathcal{B}^{(2)}_y \mid x,y\in\{1,\ldots,N\} \text{ with }A_{xy} = i\}\text{.}
	\]
	By Lemma~\ref{lem:join_monotonicity}, $\mathcal{S} = \{\mathcal{S}_1,\ldots,\mathcal{S}_N\}$ is a partition of $\mathcal{B}^{(1)} *_U \mathcal{B}^{(2)}$.
	To show that it is indeed an $(N,t_1 + t_2 + 1)$-partition, let $T$ be a $(t_1 + t_2 + 1)$-subspace of $V$ and $r = \dim((U+T)/U)$.
	Then by Lemma~\ref{lem:lambda_join}, for all $i\in\{1,\ldots,N\}$ we have
	\[
		\lambda(T,\mathcal{S}_i)
		= \sum_{(x,y) : A_{xy} = i} \lambda(U \cap T,\mathcal{B}_x^{(1)}) \cdot\lambda((U+T)/U,\mathcal{B}_y^{(2)}) \cdot q^{(u-k_1)(\bar{k}_2 - r)}\text{.}
	\]
	If $\dim(U\cap T) \leq t_1$, then $\lambda(U\cap T,\mathcal{B}_x^{(1)})$ is independent of $x$ and hence by the Latin square property
	\[
		\lambda(T,\mathcal{S}_i)
		= \lambda({U \cap T},\mathcal{B}_1^{(1)})\cdot q^{(v-k_1)(\bar{k}_2 - r)} \sum_{y=1}^N \lambda((U+T)/U,\mathcal{B}_y^{(2)})
	\]
	is independent of $i$.
	Otherwise $\dim(U\cap T) > t_1$, implying $\dim((U+T)/U) = \dim(T) - \dim(U\cap T) < t_2 + 1$, so $\lambda((U+T)/U,\mathcal{B}_y^{(2)}) = \lambda((U+T)/U,\mathcal{B}_1^{(2)})$ for all $y\in\{1,\ldots,N\}$ and therefore by the Latin square property also
	\[
		\lambda(T,\mathcal{S}_i)
		= \lambda((U+T)/U,\mathcal{B}_1^{(2)})\cdot q^{(v-k_1)(\bar{k}_2 - r)} \sum_{x=1}^N \lambda({U \cap T},\mathcal{B}_x^{(1)})
	\]
	is independent of $i$.
\end{proof}

The combination of the \enquote{$q$-Pascal decomposition} in Lemma~\ref{lem:q_pascal}\ref{lem:q_pascal:decomposition} with the theory of $(N,t)$-partitionable sets allows an alternative proof for \cite[Cor.~20]{trivial}, which also serves as a prototype for the recursive constructions of large sets we will see in Section~\ref{sect:series}.

\begin{lem}[{{\cite[Cor.~20]{trivial}}}]
\label{tvtq}
If there exists an $\LS_q[N](t,k-1,v-1)$ and an $\LS_q[N](t,k,v-1)$, then there exists an $\LS_q[N](t,k,v)$.
\end{lem}

\begin{proof}
	By Lemma~\ref{lma:ls_to_partitionable}, both $\gauss{V_{v-1}}{k-1}{q}$ and $\gauss{V_{v-1}}{k}{q}$ are $(N,t)$-partitionable.
	Furthermore, $(\mathbf{0} + V)/V$ is $(N,-1)$-partitionable.
	The application of the Basic Lemma~\ref{lem:basic} yields that $\gauss{V_{v-1}}{k-1}{q} *_{V/V_{v-1}} (\mathbf{0} + V)/V$ and $\gauss{V_{v-1}}{k}{q} *_{\overline{V/V_{v-1}}} (\mathbf{0} + V)/V$ are $(N,t+(-1)+1) = (N,t)$ partitionable.
	By Lemma~\ref{lem:q_pascal}\ref{lem:q_pascal:decomposition}, these two sets form a partition of $\gauss{V}{k}{q}$.
	Now by Lemma~\ref{lem:disjoint_union_partitionable}, $\gauss{V}{k}{q}$ is $(N,t)$-partitionable.
	Therefore by Lemma~\ref{lma:ls_to_partitionable}, there exists an $\LS_q[N](t,k,v)$.
\end{proof}

%% file: examples_halvings.tex
\section{Examples of Halvings}
\label{sect:halving_computer_constructions}
In this section, we look at large sets with the parameters $\LS_q[2](2,3,6)$, which are admissible if and only if $q$ is odd.
By Remark~\ref{rem:ls}\ref{rem:ls:halving}, such large sets correspond to subspace designs with the parameters
\[
	2\text{-}\left(6,\;3,\;\frac{1}{2}(q^2 + 1)(q + 1)\right)_q\text{,}
\]
which evaluates to $2$-$(6,3,20)_3$ for $q=3$ and $2$-$(6,3,78)_5$ for $q=5$.
In the case $q=3$, such a subspace design has been constructed in \cite{MBraun05}.
We are going to extend this result into two directions.

Denoting the $\PGL(6,3)$-image of a Singer cycle and a matching Frobenius automorphism by $\bar{\sigma}$ and $\bar{\phi}$, respectively, we will show the following counting statement in this section:

\begin{thm}
\label{thm:types_halving_q3}
Let $G = \langle\bar{\sigma}^2,\bar{\phi}^2\rangle$.
There exist exactly $57275$ isomorphism types of $G$-invariant $2$-$(6,3,20)_3$ designs.
\end{thm}

Furthermore, by constructing a $2$-$(6,3,78)_5$ subspace design, we will get:

\begin{thm}
	\label{thm:halving_example}
	There exists an $\LS_q[2](2,3,6)$ for $q\in\{3,5\}$.
\end{thm}

\subsection{The method of Kramer and Mesner}
The idea is the following~\cite{Kramer-Mesner,Miyakawa-Munemasa-Yoshiara-1995,BKL05}:
Fix some parameter set $t$-$(v,k,\lambda)_q$ and a subgroup $G$ of $\PGammaL(V)$.
Then the action of $G$ induces partitions $\gauss{V}{t}{q} = \bigcup_{i=1}^\tau \mathcal{T}_i$ and $\gauss{V}{k}{q} = \bigcup_{j=1}^\kappa \mathcal{K}_j$ into orbits.
We pick orbit representatives $T_i\in\mathcal{T}_i$ and $K_j\in\mathcal{K}_j$.
Any subspace design invariant under $G$ will have the form $(V,\mathcal{B})$ with $\mathcal{B} = \bigcup_{j\in J} \mathcal{K}_j$ and $J \subseteq \{1,\ldots,\kappa\}$.
A set $J\subseteq \{1,\ldots,\kappa\}$ induces a $t$-$(v,k,\lambda)_q$ design if and only if its characteristic vector $\chi_J \in\{0,1\}^\kappa$ is a solution of the system of linear integer equations
\[
	A \chi_J = b\text{,}
\]
where $A = (a_{ij}) \in \Z^{\kappa\times\tau}$ is the matrix with the entries
\[
    a_{ij}
    = \lambda(T_i,\mathcal{K}_j)
    = \#\{K\in\mathcal{K}_j \mid T_i \leq K\}
\]
and $b\in\Z^{\tau}$ is the vector of length $\tau$ with all entries equal to $\lambda$.
This equation system will be attacked computationally.

Thus, the method of Kramer and Mesner can be seen as kind of a trade-off:
On the one hand, the group $G$ reduces the size of the equation system, but on the other hand, we can only find subspace designs invariant under $G$ in this way.

Of course, the method can only be successful if the selected group $G$ admits a subspace design of the given parameters.
In the past, it has proven quite fruitful to prescribe certain subgroups of the $\PGL$-image of the normalizer of a Singer cycle \cite{Miyakawa-Munemasa-Yoshiara-1995,BKL05,MBraun05,SBraunDiplomarbeit,BKOW}.

Representing the $\GF(q)$-vector space $V$ of dimension $v$ as a finite field $\GF(q^v)$ and picking a primitive element $\alpha$ of $\GF(q^v)^*$, the mapping $\sigma: V\to V$, $x\mapsto \alpha x$ is in $\GL(V)$.
It is of order $q^v-1$ and an example of a \emph{Singer cycle} of $\GL(V)$.
The normalizer $N(\sigma)$ of $\sigma$ in $\GL(V)$ is given by $\langle\sigma,\phi\rangle = \langle\sigma\rangle\cdot\langle\phi\rangle$, where $\phi\in\GL(V)$ is the automorphism $x\mapsto x^q$ \cite[Satz~7.3]{Huppert-1967}.
In the case $q$ prime, $\phi$ is the Frobenius automorphism of $\GF(q^v)$.
We denote the images of $\sigma$ and $\phi$ in $\PGL(V)$ by $\bar{\sigma}$ and $\bar{\phi}$, respectively.
The image $\bar{\sigma}$ has order $\gauss{v}{1}{q}$ and $\langle\bar{\sigma}\rangle$ acts regularly on $\gauss{V}{1}{q}$.
The image $\bar{\phi}$ in $\PGL(V)$ has order $v$.

In the following, the orbit representatives $K_j$ will be given by their canonical $3\times 6$-matrices (with respect to some specified basis).
For a compact representation, the $\GF(q)$ entries will be represented by numbers in $\{0,\ldots,q-1\}$, and each row $(a_5,\ldots,a_0)$ with $a_i\in\{0,\ldots,q-1\}$ will be given in the $q$-adic representation $\sum_{i=0}^5 a_i q^i$.
So each canonical matrix is represented by a triple of numbers (one number for each row).

\subsection{The case $q=3$}
The polynomial
\[
    X^6 - X^4 + X^2  -X - 1\in\GF(3)[X]
\]
is primitive.
So a primitive element of $\GF(3^6)$ is given by any root $\alpha$, and a basis of $V = \GF(3^6)$ as a $\GF(3)$-vector space is given by $\{1,\alpha,\alpha^2,\ldots,\alpha^5\}$.
With respect to this basis, the resulting mappings $\sigma$ and $\phi$ are represented by the matrices
\[
\begin{pmatrix}
0 & 0 & 0 & 0 & 0 & 1 \\
1 & 0 & 0 & 0 & 0 & 1 \\
0 & 1 & 0 & 0 & 0 & 2 \\
0 & 0 & 1 & 0 & 0 & 0 \\
0 & 0 & 0 & 1 & 0 & 1 \\
0 & 0 & 0 & 0 & 1 & 0
\end{pmatrix}
\qquad\text{and}\qquad
\begin{pmatrix}
1 & 0 & 1 & 0 & 0 & 1 \\
0 & 0 & 1 & 1 & 1 & 2 \\
0 & 0 & 2 & 1 & 1 & 2 \\
0 & 1 & 0 & 0 & 2 & 1 \\
0 & 0 & 1 & 1 & 1 & 0 \\
0 & 0 & 0 & 0 & 2 & 2
\end{pmatrix}\text{.}
\]
The group $G = \langle \bar{\sigma}^2, \bar{\phi}^2\rangle$ has order $546$ and partitions $\gauss{V}{3}{3}$ into $2$ orbits of length $14$, $18$ orbits of length $182$ and $56$ orbits of length $546$.
A computer search showed that there are exactly $229100$ possibilities to build a $2$-$(6,3,20)_3$ design as a union of orbits.
Now we are able to show Theorem~\ref{thm:types_halving_q3}.

\begin{proof}[Proof of Theorem~\ref{thm:types_halving_q3}]
Let $P = \langle\bar{\sigma}^{52}\rangle$.
$P$ is a Sylow $7$-subgroup of $G$.
The normalizer of $P$ in $\PGL(6,3)$ is computed as $N = \langle\bar{\sigma},\bar{\phi}\rangle$.
So by \cite[Th.~3.1]{Laue}, to count the isomorphism types of $G$-invariant subspace designs, it is enough to look at the action of the subgroup $N$.

By $N/G \cong \Z/2\Z\times \Z/2\Z$ and the correspondence theorem, the subgroups of $N$ properly containing $G$ are given by $\langle \bar{\sigma},\bar{\phi}^2\rangle$, $\langle \bar{\sigma}^2,\bar{\phi}\rangle$, $\langle \bar{\sigma}^2,\bar{\sigma}\bar{\phi},\bar{\phi}^2\rangle$ and $N$.
Again using the method of Kramer and Mesner, we checked computationally that there is no $2$-$(6,3,20)_3$ design invariant under one of these groups.
So $N_{\mathcal{D}} = G$ for all $G$-invariant $2$-$(6,3,20)_3$ designs $\mathcal{D}$.
Now by the orbit-stabilizer theorem, the action of $N$ partitions the set of $G$-invariant designs (which has size $229100$ by our computer search) into orbits of size $[N : G] = 4$.
This shows that the number of isomorphism types is $229100 / 4 = 57275$.
\end{proof}

\begin{rem}
	Besides $G = \langle\bar{\sigma}^2,\bar{\phi}^2\rangle$ there is another comparably large subgroup $G'$ of $\PGL(6,3)$ admitting a $G'$-invariant $2$-$(6,3,20)_3$ design.
	It is the normalizer of a Singer cycle in $\PGL(5,3)$ of order $605$, embedded into $\PGL(6,3)$.
\end{rem}

\subsection{The case $q=5$}
Here, the primitive polynomial
\[
    X^6 + X^4 - X^3 + X^2 + 2\in\GF(5)[X]
\]
is chosen.
The resulting mappings $\sigma$ and $\phi$ are represented by the matrices
\[
\begin{pmatrix}
0 & 0 & 0 & 0 & 0 & 3 \\
1 & 0 & 0 & 0 & 0 & 0 \\
0 & 1 & 0 & 0 & 0 & 4 \\
0 & 0 & 1 & 0 & 0 & 1 \\
0 & 0 & 0 & 1 & 0 & 4 \\
0 & 0 & 0 & 0 & 1 & 0
\end{pmatrix}
\qquad\text{and}\qquad
\begin{pmatrix}
1 & 0 & 0 & 2 & 0 & 4 \\
0 & 0 & 3 & 0 & 3 & 4 \\
0 & 0 & 2 & 4 & 4 & 0 \\
0 & 0 & 4 & 4 & 1 & 3 \\
0 & 0 & 0 & 4 & 1 & 4 \\
0 & 1 & 3 & 1 & 0 & 2
\end{pmatrix}\text{.}
\]
The group $G = \langle \bar{\sigma}^2, \bar{\phi}\rangle$ has order $11718$.
The orbit sizes on $\gauss{V}{3}{5}$ are given by $(63^2 \cdot 1953^2 \cdot 3906^{24} \cdot 5859^{20} \cdot 11718^{200})$.
A solution is given by the following selection of orbits:
\begin{itemize}
	\item $1$ orbit of size $63$:\\
	{\small
	$ (3221, 728, 155)$
	}
	\item $1$ orbit of size $1953$:\\
	{\small
	$( 3133, 898, 32 )$
	}
	\item $12$ orbits of size $3906$:\\
	{\small
	$( 3144, 132, 49 )$,
	$( 627, 136, 49 )$,
	$( 3202, 631, 146 )$,
	$( 3248, 749, 246 )$, \\
	$( 3157, 662, 229 )$,
	$( 3265, 1125, 44 )$,
	$( 3224, 637, 145 )$,
	$( 3139, 647, 41 )$, \\
	$( 3643, 771, 45 )$,
	$( 3226, 739, 239 )$,
	$( 3383, 1136, 43 )$,
	$( 3263, 756, 45 )$
	}
	\item $14$ orbits of size $5859$:\\
	{\small
	$( 3224, 714, 205 )$,
	$( 3167, 629, 129 )$,
	$( 3174, 701, 242 )$,
	$( 3221, 728, 182 )$, \\
	$( 3151, 639, 132 )$,
	$( 3207, 641, 247 )$,
	$( 3220, 635, 202 )$,
	$( 3173, 736, 166 )$, \\
	$( 5629, 146, 38 )$,
	$( 3643, 1017, 26 )$,
	$( 3190, 639, 206 )$,
	$( 3227, 670, 157 )$, \\
	$( 3246, 720, 210 )$,
	$( 3127, 137, 35 )$
	}
	\item $98$ orbits of size $11718$:\\
	{\small
    $( 3262, 758, 27 )$,
    $( 3143, 749, 225 )$,
    $( 3232, 659, 198 )$,
    $( 3134, 731, 162 )$, \\
    $( 3209, 672, 165 )$,
    $( 3236, 633, 219 )$,
    $( 3194, 748, 211 )$,
    $( 3229, 669, 179 )$, \\
    $( 3381, 878, 35 )$,
    $( 3236, 698, 246 )$,
    $( 3157, 747, 138 )$,
    $( 3150, 659, 194 )$, \\
    $( 3233, 719, 223 )$,
    $( 3228, 663, 164 )$,
    $( 3207, 661, 237 )$,
    $( 4392, 144, 44 )$, \\
    $( 3130, 774, 26 )$,
    $( 3169, 642, 246 )$,
    $( 5012, 141, 41 )$,
    $( 3181, 745, 232 )$, \\
    $( 3220, 717, 148 )$,
    $( 3131, 718, 167 )$,
    $( 3233, 680, 196 )$,
    $( 3182, 702, 181 )$, \\
    $( 3649, 1138, 41 )$,
    $( 3186, 629, 161 )$,
    $( 3147, 715, 218 )$,
    $( 3156, 686, 198 )$, \\
    $( 3645, 641, 44 )$,
    $( 3510, 880, 1 )$,
    $( 3500, 636, 29 )$,
    $( 3244, 647, 129 )$, \\
    $( 3231, 699, 203 )$,
    $( 3226, 717, 228 )$,
    $( 3638, 1014, 38 )$,
    $( 3147, 696, 143 )$, \\
    $( 3245, 639, 197 )$,
    $( 3246, 718, 222 )$,
    $( 3140, 143, 31 )$,
    $( 3173, 669, 190 )$, \\
    $( 3221, 719, 161 )$,
    $( 5000, 131, 42 )$,
    $( 3513, 1145, 32 )$,
    $( 3170, 721, 241 )$, \\
    $( 3199, 714, 157 )$,
    $( 3232, 685, 201 )$,
    $( 3203, 644, 232 )$,
    $( 3223, 649, 218 )$, \\
    $( 3176, 677, 5 )$,
    $( 3167, 656, 228 )$,
    $( 3145, 888, 36 )$,
    $( 3509, 629, 33 )$, \\
    $( 3232, 694, 134 )$,
    $( 3211, 660, 207 )$,
    $( 3727, 1100, 8 )$,
    $( 3376, 954, 5 )$, \\
    $( 3274, 752, 48 )$,
    $( 3137, 670, 214 )$,
    $( 3201, 647, 210 )$,
    $( 3209, 644, 180 )$, \\
    $( 3132, 697, 160 )$,
    $( 3175, 628, 160 )$,
    $( 3154, 1000, 9 )$,
    $( 3233, 745, 159 )$, \\
    $( 3396, 1012, 48 )$,
    $( 3140, 631, 224 )$,
    $( 3153, 677, 171 )$,
    $( 3149, 718, 221 )$, \\
    $( 3380, 1139, 27 )$,
    $( 3146, 665, 242 )$,
    $( 3238, 721, 206 )$,
    $( 3225, 703, 182 )$, \\
    $( 3163, 733, 249 )$,
    $( 3227, 711, 139 )$,
    $( 3204, 704, 204 )$,
    $( 3201, 738, 163 )$, \\
    $( 3174, 725, 152 )$,
    $( 3225, 648, 223 )$,
    $( 3192, 667, 173 )$,
    $( 3140, 684, 140 )$, \\
    $( 3643, 1015, 46 )$,
    $( 3141, 636, 249 )$,
    $( 3166, 667, 202 )$,
    $( 3230, 734, 130 )$, \\
    $( 3160, 722, 218 )$,
    $( 3188, 675, 170 )$,
    $( 3219, 681, 197 )$,
    $( 3212, 662, 167 )$, \\
    $( 3230, 635, 210 )$,
    $( 3165, 715, 177 )$,
    $( 3627, 627, 6 )$,
    $( 3187, 711, 125 )$, \\
    $( 3478, 803, 8 )$,
    $( 3231, 748, 223 )$,
    $( 3131, 690, 192 )$,
    $( 3222, 625, 148 )$, \\
    $( 3504, 1003, 32 )$,
    $( 3242, 714, 226 )$
	}
\end{itemize}

\begin{rem}
	In both cases, we have checked that the prescribed group $G$ is maximal in $\langle \bar{\sigma},\bar{\phi}\rangle$ with the property that a $G$-invariant design with the parameters in question exists.
\end{rem}

%% file: infinite_series.tex
\section{Infinite two-parameter series of halvings}
\label{sect:series}
Now we are going to recursively construct infinite two-parameter families of halvings from the two halvings in Theorem~\ref{thm:halving_example}.
The strategy is to start with a suitable decomposition of $\gauss{V}{k}{q}$ into joins as discussed in Section~\ref{sect:decompositions}, and then to populate the joins with already known halvings.
The theory of Section~\ref{sect:n-t-partitionable} will imply the existence of a halving on $\gauss{V}{k}{q}$.

We start with an example to illustrate our approach.
\begin{ex}
	\label{ex:construction}
	Let $v = 10$, $k = 3$, $q \in\{3,5\}$ and $N = 2$.
	By Example~\ref{ex:decomposition_avoid_join}, a partition of $V$ is given by
	\begin{align*}
	\gauss{V}{3}{q} &
	= \gauss{V_3}{0}{q}*_{\overline{V_4/V_3}}\gauss{V/V_4}{3}{q} 
	\quad\cup\quad \gauss{V_4}{1}{q}*_{\overline{V_5/V_4}}\gauss{V/V_5}{2}{q} \\
	& \phantom{{}={}} \cup\quad \gauss{V_5}{2}{q}*_{\overline{V_6/V_5}}\gauss{V/V_6}{1}{q}
	\quad\cup\quad \gauss{V_6}{3}{q}*_{\overline{V_7/V_6}}\gauss{V/V_7}{0}{q}\text{.}
	\end{align*}
	In the following, we keep only the relevant information of this formula, meaning that we drop the factors of the avoiding join, replace the involved vector spaces by their dimension and remove the subscript~$q$.
	In this sense, the above formula reduces to the \emph{decomposition type}
	\[
		\gaussm{10}{3} = \gaussm{3}{0}*\gaussm{6}{3} \;\cup\; \gaussm{4}{1}*\gaussm{5}{2} \;\cup\; \gaussm{5}{2}*\gaussm{4}{1} \;\cup\; \gaussm{6}{3}*\gaussm{3}{0}\text{.}
	\]
	By Theorem~\ref{thm:halving_example}, $\gaussm{6}{3}$ is $(2,2)$-partitionable.
	By considering derived large sets (see Theorem~\ref{lem:resLS}), $\gaussm{5}{2}$ is $(2,1)$-partitionable and $\gaussm{4}{1}$ is $(2,0)$-partitionable.
	Moreover, $\gaussm{3}{0}$ is $(2,-1)$-partitionable, of course.

	Now by the Basic Lemma~\ref{lem:basic}, $\gaussm{3}{0}*\gaussm{6}{3}$ is $(2,(-1) + 2 + 1) = (2,2)$-partitionable.
	Similarly, $\gaussm{4}{1}*\gaussm{5}{2}$, $\gaussm{5}{2}*\gaussm{4}{1}$ and $\gaussm{6}{3}*\gaussm{3}{0}$ are $(2,2)$-partitionable, too.
	Finally by Lemma~\ref{lem:disjoint_union_partitionable}, $\gaussm{10}{3}$ is $(2,2)$-partitionable, so by Lemma~\ref{lma:ls_to_partitionable}, there exists an $\LS_q[2](2,3,10)$.
\end{ex}

\begin{lem}
	\label{lma:one_parameter_recursion}
	If there exists an $\LS_q[N](2,3,6)$, then there exists an $\LS_q[N](2,3,v)$ for all integers $v\geq 6$ with $v\equiv 2\mod* 4$.
\end{lem}

\begin{proof}
	We proceed by induction over $v$.
	Let $v \geq 10$ be an integer with $v\equiv 2\mod* 4$.
	By $k = s = 3$ in Theorem~\ref{thm:avoidpartition} and the same notational convention as in Example~\ref{ex:construction}, we get the decomposition type
	\[
		\gaussm{v}{3} = \gaussm{3}{0}*\gaussm{v-4}{3} \;\cup\; \gaussm{4}{1}*\gaussm{v-5}{2} \;\cup\; \gaussm{5}{2}*\gaussm{v-6}{1} \;\cup\; \gaussm{6}{3}*\gaussm{v-7}{0}\text{.}
	\]
	into avoiding joins.
	By the induction hypothesis, $\gaussm{v-4}{3}$ and $\gaussm{6}{3}$ are $(N,2)$-partitionable.
	From the derived large sets, we get that $\gaussm{v-5}{2}$ and $\gaussm{5}{2}$ are $(N,1)$-partitionable and that $\gaussm{v-6}{1}$ and $\gaussm{4}{1}$ are $(N,0)$-partitionable.
	Furthermore, $\gaussm{v-7}{0}$ and $\gaussm{3}{0}$ are $(N,-1)$-partitionable.
	Now by the Basic Lemma~\ref{lem:basic} and Lemma~\ref{lem:disjoint_union_partitionable}, $\gaussm{v}{3}$ is $(N,2)$-partitionable and thus by Lemma~\ref{lma:ls_to_partitionable}, there exists an $\LS_q[N](2,3,v)$.
\end{proof}

It is easily checked that a halving $\LS_q[2](2,3,6)$ is admissible if and only if $q$ is odd.
In this case, $\LS_q[2](2,3,v)$ is admissible if and only if $v \geq 6$ and $v\equiv 2\mod* 4$.
So by Theorem~\ref{thm:halving_example} and Lemma~\ref{lma:one_parameter_recursion}, we get:

\begin{cor}
	For $q\in\{3,5\}$, the large set parameters $\LS_q[2](2,3,v)$ are admissible if and only they are realizable.
\end{cor}

\begin{thm}
	\label{thm:two_parameter_recursion}
	If there exists an $\LS_q[N](2,3,6)$, then there exists an $\LS_q[N](2,k,v)$ for all integers $v$ and $k$ with $v\geq 6$, $v\equiv 2\mod* 4$, $3\leq k \leq v-3$ and $k\equiv 3\mod* 4$.
\end{thm}

\begin{proof}
	Let $v \geq 6$ be an integer with $v\equiv 2\mod* 4$.
	We proceed by induction over $k$, considering all $v$ simultaneously.
	The base case $k = 3$ was shown in Lemma~\ref{lma:one_parameter_recursion}.
	Now let $k\in\{7,\ldots,v-3\}$ with $k\equiv 3\mod* 4$.
	Plugging $s = 3$ into Theorem~\ref{thm:coverpartition}, we get the decomposition type
	\begin{align}
		\label{eq:two_parameter_recursion:decomposition}
		\gaussm{v}{k} & = \bigcup_{i=0}^{v-k} \gaussm{v-4-i}{k-4} * \gaussm{3+i}{3}
	\end{align}
	into covering joins.

	For $i \equiv 0\mod* 4$, $v - 4 - i \equiv v \equiv 2\mod* 4$ and $k-4\equiv k \equiv 3\mod* 4$, so $\gaussm{v-4-i}{k-4}$ is $(N,2)$-partitionable by the induction hypothesis.
	Besides that, $\gaussm{3+i}{3}$ is $(N,-1)$-partitionable, so by the Basic Lemma~\ref{lem:basic}, $\gaussm{v-4-i}{k-4} * \gaussm{3+i}{3}$ is $(N,2)$-partitionable for all $i\in\{0,\ldots,v-k\}$ with $i\equiv 0\mod* 4$. 

	For $i \equiv 1\mod* 4$, $v - 4 - i \equiv 1\mod* 4$, so $v - 3 - i\equiv 2\mod* 4$.
	By the induction hypothesis, there exists an $\LS_q[N](2,k-4,v-3-i)$.
	Taking the residual, there exists an $\LS_q[N](1,k-4,v-4-i)$.
	So $\gaussm{v-4-i}{k-4}$ is $(N,1)$-partitionable.
	Furthermore, by $5+i \equiv 2\mod* 4$ and the induction hypothesis there exists an $\LS_q[N](2,3,5+i)$.
	Taking the residual twice, we get an $\LS_q[N](0,3,3+i)$.
	The Basic Lemma~\ref{lem:basic} implies that $\gaussm{v-4-i}{k-4} * \gaussm{3+i}{3}$ is $(N,2)$-partitionable for all $i\in\{0,\ldots,v-k\}$ with $i\equiv 1\mod* 4$.

	Similarly, we see that $\gaussm{v-4-i}{k-4} * \gaussm{3+i}{3}$ is $(N,2)$-partitionable for all $i\in\{0,\ldots,v-k\}$ with $i\equiv 2,3\mod* 4$, too, see Table~\ref{tbl:two_parameter_recursion}.
	\begin{table}
	\caption{Analysis of the decomposition in the proof of Theorem~\ref{thm:two_parameter_recursion}}
	\label{tbl:two_parameter_recursion}
	\noindent\centering
		$\begin{array}{cccc}
			i                & \gaussm{v-4-i}{k-4}        & \gaussm{3+i}{3} & \gaussm{v-4-i}{k-4} * \gaussm{3+i}{3} \\
			\hline
			i\equiv 0\mod* 4 & (N,2)\text{-part.} & (N,-1)\text{-part.} & (N,2)\text{-part.} \\
			i\equiv 1\mod* 4 & (N,1)\text{-part.} & (N,0)\text{-part.} & (N,2)\text{-part.} \\
			i\equiv 2\mod* 4 & (N,0)\text{-part.} & (N,1)\text{-part.} & (N,2)\text{-part.} \\
			i\equiv 3\mod* 4 & (N,-1)\text{-part.} & (N,2)\text{-part.} & (N,2)\text{-part.} 
		\end{array}$
	\end{table}
	Now the application of Lemma~\ref{lem:disjoint_union_partitionable} to the decomposition~\eqref{eq:two_parameter_recursion:decomposition} and Lemma~\ref{lma:ls_to_partitionable} yields the existence of an $\LS_q[N](2,k,v)$.
\end{proof}

\begin{rem}
\item
\begin{enumerate}[(i)]
\item Based on the fairly general machinery of $(N,t)$-partitionable sets in Section~\ref{sect:n-t-partitionable}, one can create more statements in the style of Theorem~\ref{thm:two_parameter_recursion} by starting with a suitable decomposition of $V$ into joins.
The problem is that always some example of a large set is needed as a starting point for the recursion.
Our above decomposition was tailored to fit the large set parameters $\LS_q[2](2,3,6)$ discussed in Section~\ref{sect:n-t-partitionable}.
\item The parameter set in Theorem~\ref{thm:two_parameter_recursion} is closed under taking duals:
If $v$ and $k$ are integers with $v\geq 6$, $v\equiv 2\mod* 4$, $3\leq k \leq v-3$ and $k\equiv 3\mod* 4$, then the dual of an $\LS_q[N](2,k,v)$ is an $\LS_q[N](2,v-k,v)$ with $3\leq v-k \leq v-3$ and $v-k\equiv 2-3\equiv 3\mod* 4$.
\end{enumerate}
\end{rem}

\begin{cor}
	\label{cor:halving_series}
	Let $q\in\{3,5\}$, and $v,k$ be integers with $v\geq 6$, $v\equiv 2\mod* 4$, $3\leq k\leq v-3$ and $k\equiv 3\mod* 4$.
	\begin{enumerate}[(a)]
		\item\label{cor:halving_series:halving} There exists a halving $\LS_q[2](2,k,v)$.
		\item\label{cor:halving_series:designs} There exists a $2$-$\left(v,k,\gauss{v-2}{k-2}{q}/2\right)_q$ subspace design.
	\end{enumerate}
\end{cor}

\begin{proof}
	For part~\ref{cor:halving_series:halving}, apply Theorem~\ref{thm:two_parameter_recursion} to the halvings in Theorem~\ref{thm:halving_example}.
	The design parameters of the halvings are given in part~\ref{cor:halving_series:designs} .
\end{proof}

Our knowledge for the existence of $\LS_q[2](2,k,v)$ for $q\in\{3,5\}$ is shown in Table~\ref{tbl:ueberblick}. 
A minus sign indicates that the parameters are not admissible, and a question mark that the parameters are admissible, but the realizability is open.
In the case that Corollary~\ref{cor:halving_series}\ref{cor:halving_series:halving} yields the existence, the parameter $k$ is displayed.
Because of duality, only the parameter range $3\leq k \leq v/2$ is shown.

\begin{table}
\caption{Admissibility and realizability of $\LS_q[2](2,k,v)$, $q\in\{3,5\}$}
\label{tbl:ueberblick}
\noindent\centering\begin{tiny} 
\begingroup
\setlength\tabcolsep{0pt}
 \begin{tabular}{*{33}{>{\centering\arraybackslash}p{2.4mm}}@{\hskip 2mm}|@{\hskip 1mm}c@{\hskip 1mm}|} 
 & & & & & & & & & & & & & & & & & & & & & & & & & & & & & & & & & $\mathbf{v}$ \\
 & & & & & & & & & & & & & & & & & & & & & & & & & & & & & & & &3& 6 \\
 & & & & & & & & & & & & & & & & & & & & & & & & & & & & & & &-& & 7 \\
 & & & & & & & & & & & & & & & & & & & & & & & & & & & & & &-& &-& 8 \\
 & & & & & & & & & & & & & & & & & & & & & & & & & & & & &-& &-& & 9 \\
 & & & & & & & & & & & & & & & & & & & & & & & & & & & &3& &?& &?& 10 \\
 & & & & & & & & & & & & & & & & & & & & & & & & & & &-& &?& &?& & 11 \\
 & & & & & & & & & & & & & & & & & & & & & & & & & &-& &-& &?& &?& 12 \\
 & & & & & & & & & & & & & & & & & & & & & & & & &-& &-& &-& &?& & 13 \\
 & & & & & & & & & & & & & & & & & & & & & & & &3& &-& &-& &-& &7& 14 \\
 & & & & & & & & & & & & & & & & & & & & & & &-& &-& &-& &-& &-& & 15 \\
 & & & & & & & & & & & & & & & & & & & & & &-& &-& &-& &-& &-& &-& 16 \\
 & & & & & & & & & & & & & & & & & & & & &-& &-& &-& &-& &-& &-& & 17 \\
 & & & & & & & & & & & & & & & & & & & &3& &?& &?& &?& &7& &?& &?& 18 \\
 & & & & & & & & & & & & & & & & & & &-& &?& &?& &?& &?& &?& &?& & 19 \\
 & & & & & & & & & & & & & & & & & &-& &-& &?& &?& &?& &?& &?& &?& 20 \\
 & & & & & & & & & & & & & & & & &-& &-& &-& &?& &?& &?& &?& &?& & 21 \\
 & & & & & & & & & & & & & & & &3& &-& &-& &-& &7& &?& &?& &?& &11& 22 \\
 & & & & & & & & & & & & & & &-& &-& &-& &-& &-& &?& &?& &?& &?& & 23 \\
 & & & & & & & & & & & & & &-& &-& &-& &-& &-& &-& &?& &?& &?& &?& 24 \\
 & & & & & & & & & & & & &-& &-& &-& &-& &-& &-& &-& &?& &?& &?& & 25 \\
 & & & & & & & & & & & &3& &?& &?& &?& &7& &-& &-& &-& &11& &?& &?& 26 \\
 & & & & & & & & & & &-& &?& &?& &?& &?& &-& &-& &-& &-& &?& &?& & 27 \\
 & & & & & & & & & &-& &-& &?& &?& &?& &-& &-& &-& &-& &-& &?& &?& 28 \\
 & & & & & & & & &-& &-& &-& &?& &?& &-& &-& &-& &-& &-& &-& &?& & 29 \\
 & & & & & & & &3& &-& &-& &-& &7& &-& &-& &-& &11& &-& &-& &-& &15& 30 \\
 & & & & & & &-& &-& &-& &-& &-& &-& &-& &-& &-& &-& &-& &-& &-& & 31 \\
 & & & & & &-& &-& &-& &-& &-& &-& &-& &-& &-& &-& &-& &-& &-& &-& 32 \\
 & & & & &-& &-& &-& &-& &-& &-& &-& &-& &-& &-& &-& &-& &-& &-& & 33 \\
 & & & &3& &?& &?& &?& &7& &?& &?& &?& &11& &?& &?& &?& &15& &?& &?& 34 \\
 & & &-& &?& &?& &?& &?& &?& &?& &?& &?& &?& &?& &?& &?& &?& &?& & 35 \\
 & &-& &-& &?& &?& &?& &?& &?& &?& &?& &?& &?& &?& &?& &?& &?& &?& 36 \\
 &-& &-& &-& &?& &?& &?& &?& &?& &?& &?& &?& &?& &?& &?& &?& &?& & 37 \\
3& &-& &-& &-& &7& &?& &?& &?& &11& &?& &?& &?& &15& &?& &?& &?& &19& 38 \\
\end{tabular} 
\endgroup
 \end{tiny} 
 \end{table}

\begin{table}
	\caption{Parameters of small halvings in Corollary~\ref{cor:halving_series}}
	\label{tbl:smallest_halvings}
	\noindent\resizebox{\linewidth}{!}{
		$\begin{array}{lllll}
		q & v & k & \lambda & \text{size} \\
		\hline
		3 & 6 & 3 & 20 & 16940 \\
		3 & 10 & 3,7 & 1640 & 9163363880 \\
		3 & 14 & 3,10 & 132860 & 4870846320040820 \\
		3 & 14 & 7 & 44558972694792920 & 213514388484588339982040 \\
		5 & 6 & 3 & 78 & 1279278 \\
		5 & 10 & 3,7 & 48828 & 312943420103028 \\
		5 & 14 & 3,10 & 30517578 & 76402444317336044321778 \\
		5 & 14 & 7 & 1913728386070579497083028 & 11681368214414934224094848999708028
		\end{array}$
	}
\end{table}

\begin{ex}
	In Table~\ref{tbl:smallest_halvings}, the halvings produced by Corollary~\ref{cor:halving_series} are listed up to $v = 14$, together with the $\lambda$-value and the size of the corresponding subspace designs.
\end{ex}

Finally, it is worth noting

\begin{cor}
	There are infinitely many nontrivial large sets of subspace designs with $t = 2$.
\end{cor}

%% file: conclusion.tex
\section{Conclusion}
\label{sect:conclusion}
We conclude this article with a few open questions arising from the present work.
\begin{enumerate}[(i)]
	\item The two halvings $\LS_q[2](2,3,6)$ in Theorem~\ref{thm:halving_example} have been constructed by computer.
	Give a computer-free construction of those large sets.
	Can we find such a construction for every odd prime power $q$?
	
	\item For any prime power $q$, the parameters $\LS_q[q^2 + 1](2,3,6)$ are admissible in the sense of Lemma~\ref{lem:ls_integrality}.
	Can those large sets be realized for certain (all?) values of $q$?
	So far, not a single example with those parameters is known.
	For every such large set, Theorem~\ref{thm:two_parameter_recursion} would give an infinite two-parameter series of large sets.
	Furthermore, for odd prime powers $q$ the existence of an $\LS_q[q^2 + 1](2,3,6)$ would imply the existence of a halving $\LS_q[2](2,3,6)$ by Lemma~\ref{lem:ls_divisorsN}.

	\item For large sets of classical block designs ($q = 1$), the \emph{halving conjecture} states that all admissible halvings are realizable \cite[Sect.~5]{Hartman}.
	For $t = 2$, the conjecture has been proven in \cite{Ajoodani98}.
	Can anything be said about this conjecture in the $q$-analog case?

	\item In the recursion techniques of Section~\ref{sect:n-t-partitionable}, the parameters $q$ and $N$ are constants.
	Find recursion techniques which alter those values in a nontrivial way.
\end{enumerate}